\documentclass[a4paper,10pt]{amsart}
\usepackage{a4wide}
\usepackage[dvips]{graphicx}
\usepackage{color}
\usepackage{amssymb}

\usepackage{amsmath}
\usepackage{amsfonts}
\usepackage{latexsym}
\usepackage{amscd}
\usepackage{graphicx}
\usepackage[latin1]{inputenc}
\usepackage[english]{babel}
\usepackage{enumerate}
\usepackage{afterpage}
\def\a{\alpha}
\def\b{\beta}
\def\d{\delta}

\def\g{\gamma}

\def\ds{\displaystyle}

\newcommand{\F}{\mathcal{F}}
\newcommand{\N}{\mathbb{N}}

\renewcommand{\to}{\longrightarrow}
\def\co{\colon\thinspace}
\newcommand{\Um}{\boldsymbol{\mu}}
\newcommand{\Ut}{\boldsymbol{\tau}}

\newcommand{\C}{\mathbb{C}}
\newcommand{\Z}{\mathbb{Z}}
\newcommand{\Q}{\mathbb{Q}}
\newcommand{\R}{\mathbb{R}}
\newcommand{\PSL}{\mathrm{PSL}}
\newcommand{\SLtwoC}{\mathrm{SL}(2,\C)}
\newcommand{\SLtwoR}{\mathrm{SL}(2,\R)}

\newcommand{\PSLtwoC}{\mathrm{PSL}(2,\C)}
\newcommand{\PSLtwoR}{\mathrm{PSL}(2,\R)}

\newcommand{\Tr}{\mbox{tr }}

\newcommand{\MCG}{\mathcal{MCG}}

\newcommand{\HH}{{\mathbb H}^2}

\newcommand{\X}{\mathcal{X}}

\newcommand{\Sc}{\mathcal{S}}
\newcommand{\tr}{\mathrm{tr}}

\newtheorem{Theorem}{Theorem}[section]
\newtheorem{Lemma}[Theorem]{Lemma}
\newtheorem{Proposition}[Theorem]{Proposition}
\newtheorem{Corollary}[Theorem]{Corollary}
\newtheorem{introthm}{Theorem}

\newtheorem{Definition}[Theorem]{Definition}
\newtheorem{Remark}[Theorem]{Remark}
\newtheorem{Example}[Theorem]{Example}
\newtheorem{Notation}[Theorem]{Notation}

\begin{document}

\title{On the character variety of the four--holed sphere}

\author{Sara Maloni}
\address{D\'{e}partement de Math\'{e}matiques de la Facult\'{e} des Sciences d'Orsay, Universit\'{e} Paris-Sud 11}
\email{sara.maloni@math.u-psud.fr}
\urladdr{www.math.u-psud.fr/$\sim$maloni}

\author{Fr\'{e}d\'{e}ric Palesi}
\address{Laboratoire d'Analyse, Topologie et Probabilit\'es (LATP), Aix-Marseille Universit\'{e}}
\email{frederic.palesi@univ-amu.fr}
\urladdr{www.latp.univ-mrs.fr/$\sim$fpalesi}

\author{Ser Peow Tan}
\address{Department of Mathematics, National University of Singapore}
\email{mattansp@nus.edu.sg}
\urladdr{http://www.math.nus.edu.sg/$\sim$mattansp}

\thanks{The first author was partially supported by the European Research Council under the {\em European Community}'s seventh Framework Programme (FP7/2007-2013)/ERC {\em grant agreement} n FP7-246918. The second author is partially supported by the ANR 2011 BS 01 020 01 ModGroup. The third author was supported by the National University of Singapore academic research grant R-146-000-156-112.}

\begin{abstract}
We study the (relative) $\SLtwoC$ character varieties of the four-holed sphere and the action of the mapping class group on it. We describe a domain of discontinuity for this action, and, in the case of real characters, show that this domain of discontinuity may be non-empty on the components where the relative euler class is non-maximal.
\end{abstract}

\maketitle

\section{Introduction}\label{s:intro}

In his PhD thesis \cite{mcs_are}, McShane established the following remarkable identity for lengths of simple closed geodesics on a once-punctured torus $S_{1,1}$ with a complete, finite area hyperbolic structure:
\begin{equation}\label{McShane}
  \sum_\g \frac{1}{1+\exp(l(\g))}=\frac{1}{2},
\end{equation}
where $\g$ varies over all simple closed geodesics on $S_{1,1}$, and $l(\g)$ is the hyperbolic length of $\g$ under the given hyperbolic structure on $S_{1,1}$. This result was later generalized to (general) hyperbolic surfaces with cusps by McShane himself \cite{mcs_sim}, to hyperbolic surfaces with cusps and/or geodesic boundary components by Mirzakhani \cite{mir_sim}, and to hyperbolic surfaces with cusps, geodesic boundary and/or conical singularities, as well as to classical Schottky groups by Tan, Wong and Zhang in \cite{tan_gen2}, \cite{tan_mcs}.

On the other hand, Bowditch in \cite{bow_apr} gave an alternative proof of (\ref{McShane}) via Markoff maps, and extended it in \cite{bow_mar} to type-preserving representations of the once-punctured torus group into $\SLtwoC$ satisfying certain conditions which we call here the BQ--conditions (Bowditch's Q--conditions). He also obtained in \cite{bow_ava} a variation of (\ref{McShane}) which applies to hyperbolic once-punctured torus bundles. Subsequently, Sakuma \cite{sak_var}, Akiyoshi, Miyachi and Sakuma \cite{aki_are}, \cite{aki_var} and recently  Sakuma and Lee \cite{lee_ava}  refined Bowditch's results and generalized them to those which apply to hyperbolic punctured surface bundles. In \cite{tan_gen} Tan, Wong and Zhang also further extended  Bowditch's results to representations of the once-punctured torus group into $\SLtwoC$ which are not type-preserving, that is, where the commutator is not parabolic, and also to representations which are fixed by an Anosov element of the mapping class group and which satisfy a relative version of the Bowditch's Q--conditions. They also showed that the BQ-conditions defined an open subset of the character variety on which the mapping class group of the punctured torus acted properly discontinuously.

In a different direction, Labourie and McShane in \cite{lab_cro} showed that the identity above has a natural formulation in terms of (generalised) cross ratios, and then, using this formulation, studied identities arising from the cross ratios constructed by Labourie for representations from fundamental groups of surfaces to $\PSL(n, \R)$.

\medskip

The above papers provided much of the motivation for this paper, in particular, the identities obtained were in many cases valid for the moduli spaces of hyperbolic structures, so invariant under the action of the mapping class group, and in the case of cone structures, they could be interpreted as identities valid for certain subsets of the character variety which were invariant under the action of the mapping class group, even though the representations in the subset may be non-discrete or non-faithful. This leads naturally to the question of whether there were interesting subsets of the character varieties on which the mapping class group acts properly discontinuously, but which consists of more than just discrete, faithful representations, as explored in the punctured torus case in \cite{tan_gen}.

\medskip

In this paper we will consider representations of the free group on three generators
$F_3=\langle \alpha,\beta,\gamma,\delta : \alpha \beta \gamma \delta=I \rangle$
into $\SLtwoC$. We adopt the viewpoint that $F_3$ is the fundamental group of the four-holed sphere $S$, with $\alpha, \beta, \gamma, \delta$ identified with $\partial S$, and study the natural action of $\MCG(S)$, the mapping class group of $S$ on the character variety
$$\X:=\mathrm{Hom}(F_3, \SLtwoC)//\SLtwoC,$$ where we take the quotient in the sense of geometric invariant theory. If $\theta \in \MCG(S)$ and $[\rho] \in \X$, this action is given by $$\theta([\rho])=[\rho \circ (\theta_*)^{-1}],$$ where $\theta_*\co \pi_1(S) \to \pi_1(S)$ is the map associated to $\theta$ in homotopy.
We are interested in the dynamics of this action, in particular, on the relative character varieties $\X_{(a,b,c,d)}$, which is the set of representations for which the traces of the boundary curves are fixed.

We describe the following result, see Theorems \ref{thm:openproperty} and \ref{thm:properlydiscontinuously}.

\begin{introthm} \label{thmA}
 There exists a domain of discontinuity for the action of $\MCG(S)$ on $\X_{(a,b,c,d)}$, that is, an open subset $D \subset \X_{(a,b,c,d)}$ on which $\MCG(S)$ acts properly discontinuously.
\end{introthm}

This set is described by two conditions, much in the spirit of \cite{bow_mar} and \cite{tan_gen} given as follows. If $\mathcal{S}$ denotes the set of free homotopy classes of essential, non-peripheral simple closed curves on $S$, then the conditions for $[\rho]$ to be in $D$  are
\begin{enumerate}
  \item [(i)] $\Tr \rho(\gamma) \not\in [-2,2]$ for all $\gamma \in \mathcal{S}$; and
  \item [(ii)] $|\Tr \rho(\gamma)|<K$ for only finitely many $\gamma \in \mathcal{S}$, where  $K>0$ is a fixed constant that depends  only on $a,b,c,d$.
\end{enumerate}
Furthermore, the set of $\gamma$ satisfying condition (ii) above satisfy a quasi-convexity property, equivalently, is connected when represented as the subset of the complementary regions of a properly embedded binary tree (see Theorem \ref{thm:omega}). This property is particularly important when writing a computer program to draw slices of the domain of discontinuity.

As already observed by several other authors in related situations (see Goldman \cite{gol_the}, Tan--Wong--Zhang \cite{tan_gen} and Minsky \cite{min_the}), our domain of discontinuity contains the discrete faithful characters, but also characters which may not be discrete or faithful.

Of particular interest is the set of real characters, which consists of representations in $\SLtwoR$ or $\mathrm{SU}(2)$. In the latter case, Goldman \cite{gol_erg} proved ergodicity of the mapping class group action for all orientable hyperbolizable surfaces. (This was generalized by the second author in the non-orientable case in \cite{pal_erg}). On the other hand, in the $\SLtwoR$ case the dynamics is much richer and less understood. For example, when $S_g$ is a closed surface of genus $g \ge 2$, Goldman conjectured that the action of $\MCG(S_g)$ on the components of $\X(S_g)$ with non-maximal Euler class is ergodic. An approach towards a proof of this would be to use a cut-and-paste argument involving pieces homeomorphic to one-holed tori and four-holed spheres.  While the case of the one-holed torus was completely described by Goldman in \cite{gol_the}, we obtain partial results in the four-holed sphere case here. In fact, an important corollary of our analysis is the following (see Theorem \ref{thm:realnonemptydomain}):

\begin{introthm} \label{thmB}
The components of the relative $\SLtwoR$ character variety of $S$ with non-maximal Euler class may contain non-empty domains of discontinuity for the $\MCG(S)$ action.
\end{introthm}

This implies that there are representations in these components for which all essential simple closed curves on $S$ have hyperbolic representatives, even though these representations may not be discrete and faithful. There are also some surprises here, in particular, certain slices of the real character variety satisfying some general condition always have non-empty intersection with the domain of discontinuity.

The mapping class group $\MCG(S)$ consisting of equivalence classes of diffeomorphisms of $S$ fixing the boundary is isomorphic to the index two normal subgroup of the triangle group $\Z_2 \ast \Z_2 \ast \Z_2$ generated by reflections on an ideal triangle and it acts on the character variety as given in the earlier part of the introduction.
 For a representation $\rho : F_3 \rightarrow \SLtwoC$, we can look at its character
$$\begin{array}{rcl}
\chi_{\rho} : F_3 & \longrightarrow & \C \\ W & \longmapsto & \tr (\rho (W))
\end{array}$$
The character variety is exactly the set of characters,  by results of \cite{mag_rin}, see for example \cite{gol_the}, each character is in turn  determined entirely by its value on the seven elements as follows, satisfying a single equation:
$$\chi_{\rho} (\alpha):=a, \quad \chi_{\rho} (\beta):=b, \quad \chi_{\rho} (\gamma):=c, \quad \chi_{\rho} (\delta):=d,$$
 $$\chi_{\rho} (\alpha \beta):=x, \quad \chi_{\rho} (\beta \gamma):=y, \quad \chi_{\rho} (\gamma \alpha) :=z.$$
Hence, we identify $\X$ with the variety $\mathcal{V}$ consisting of points $(a,b,c,d,x,y,z) \in \C^7$ satisfying the equation
$$x^2+y^2+z^2+xyz=px+qy+rz+s,$$ where
$$p=ab+cd, \quad q=bc+ad, \quad r=ac+bd, \quad s=4-a^2-b^2-c^2-d^2-abcd.$$
The  $\MCG(S)$  action extends to an action of $\Z_2 \ast \Z_2 \ast \Z_2$ on $\mathcal{V}$, which is generated by $$\theta_1, \theta_2, \theta_3\co \mathcal{V} \rightarrow \mathcal{V},$$
where
\begin{eqnarray*}
  \theta_1(a,b,c,d,x,y,z)&=&(a,b,c,d, p-yz-x, y,z)\\
   \theta_2(a,b,c,d,x,y,z)&=&(a,b,c,d, x, q-xz-y,z)\\
   \theta_3(a,b,c,d,x,y,z)&=&(a,b,c,d, x, y,r-xy-z).\\
\end{eqnarray*}
Elements of the mapping class group correspond to words of even length in $\theta_1, \theta_2$ and $\theta_3$, in particular, $\theta_2\theta_3$, $\theta_3\theta_1$ and $\theta_1\theta_2$ correspond to Dehn twists about essential simple closed curves on $S$ and generate the $\MCG(S)$ action on $\mathcal{V}$.

\medskip

It is somewhat remarkable that many of the results of \cite{bow_mar} and \cite{tan_gen} generalize to the problem we study here, although the analysis is necessarily more technical and complicated, but in some sense, also more interesting. There are however some results which do not generalize, see  Example  \ref{rmk:raywithfiniteintersection}. Finally, we note that recent work of Hu, Tan and Zhang \cite{hut_cox} on Coxeter group actions on quartic varieties indicate that in fact, there should be a deeper underlying theory for analyzing the domains of discontinuity for group actions of this type.

\vskip 5pt

The rest of this paper is organized as follows. In \S \ref{s:not} we set the notation and give some basic definitions. In \S \ref{s:markoff} we prove the generalizations of the basic lemmas (in terms of Markoff maps) required to analyse and understand the orbit of a character under the action of the mapping class group. In particular,  generalizations of the ``fork'' lemma (Lemma \ref{lem:fork}) and the quasi-convexity result (Theorem \ref{thm:omega}) from \cite{bow_mar} and \cite{tan_gen}, as well as an analysis of the values taken by the neighbors around a region are covered. In \S \ref{s:domain} we give a proof of our main theorem (Theorem \ref{thmA}) which describes the domain of discontinuity for the action in terms of the BQ-conditions. In \S \ref{s:real_case}, we consider the real case and show that domains of discontinuity can occur in the components of the $\PSLtwoR$  relative character variety which do not have maximal relative Euler class (Theorem \ref{thmB}). Finally, in \S \ref{s:Conclusion} we give some concluding remarks.

\vskip 10pt

\noindent {\it Acknowledgements.} This project was initiated when the authors were participating in the trimester program on ``Geometry and Analysis of surface group representations'' at the Institut Henri Poincar\'e (Jan-Mar 2012). The authors are grateful to the organizers of the program for the invitations to participate in the program, and to the IHP and its staff for their hospitality and generous support. In particular, we would like to thank Bill Goldman for helpful conversations.

\section{Notation}\label{s:not}

In this section we set the notation which we will use in  the paper and  give some important definitions. Since many of the results included in this article are influenced by, and are generalisations of the results of Tan, Wong and Zhang's article \cite{tan_gen}, we will try to follow closely the notation and structure of that paper, so that the interested reader will find it easier to compare our results with theirs. We should also note that their paper was a generalisation of Bowditch's results \cite{bow_mar}.

\subsection{The four-holed sphere group}

Let $S = S_{0,4}$ be a (topological) four-holed sphere, that is, a sphere with four open disks removed, and let $\Gamma$ be its fundamental group. The group $\Gamma$ admits the following presentation
$$\Gamma = \langle \a, \b, \g, \d | \a\b\g\d \rangle , $$
where $\a$, $\b$, $\g$ and $\d$ correspond to homotopy classes of the four boundary components, one for each removed disk. Note that $\Gamma$ is isomorphic to the free group on three generators $\Z \ast \Z \ast \Z = \langle \a, \b, \g \rangle$.

We define an equivalence relation $\sim$ on $\Gamma$ by: $g \sim h$ if and only if $g$ is conjugate to $h$ or $h^{-1}$. Note that $\Gamma/\!\sim$ can be identified with the set of free homotopy classes of unoriented closed curves on $S$.

\subsection{Simple closed curves on the sphere}

Let $\Sc = \Sc(S)$ be the set of free homotopy classes of essential (that is, non-trivial and non-peripheral) simple closed curves on $S$ and let $\hat \Omega \subset \Gamma/\!\sim$ be the subset corresponding to $\Sc$. Note that $\hat \Omega$ can be identified with $\hat\Q = \Q \cup \{\infty\}$ by considering the `slope' of $[g] \in \hat \Omega$, see, among others, Proposition 2.1 of Keen and Series \cite{kee_the2}. For example, we can identify $\a\b$ with $0$, $\b\g$ with $\infty$, $\a\g$ with $-1$, and so on.

We also observe that $\hat \Omega$ inherits a cyclic ordering from the cyclic ordering of $\hat \Q$ induced from the standard embedding into $\hat \R = \R \cup \{\infty\} \cong S^1$.

\subsection{Relative character variety of $\Gamma$}

The character variety ${\mathcal X} = {\mathcal X}(\Gamma, \SLtwoC)$ is the space of equivalence classes of representations $\rho \co \Gamma \rightarrow \SLtwoC$, where the equivalence classes are obtained by taking the closure of the orbit under the conjugation action by $\SLtwoC$.

A representation $\rho\co \Gamma \to \SLtwoC$ is said to be a $\Ut$--{\it representation}, or $\Ut$--{\it character}, where $\Ut = (a,b,c,d) \in \C^4$, if for some fixed generators $\a, \b, \g, \d \in \Gamma$ corresponding to the boundary components of $S$, $\tr \rho(\a) = a$, $\tr \rho(\b) = b$, $\tr \rho(\g) = c$, $\tr \rho(\d) = d$. The space of equivalence classes of $\Ut$-representations is denoted by ${\mathcal X}_{\Ut}$ and is called the $\Ut$--\emph{relative character variety}. These representations correspond to representations of the four-holed sphere where we fix the conjugacy classes of the four boundary components.

For $\rho \in \mathcal{X}_{\Ut} $, we let $x= \tr( \rho (\a \b )) $, $y = \tr (\rho (\b \g )) $ and $z = \tr ( \rho (\a \g )) $. A classical result on the character varieties (see (9) in p. 298 of Fricke and Klein \cite{fri_vor}) states that $\mathcal{X}_{\Ut}$ is identified with the set
    \begin{align*}
        \left\{ (x,y,z) \in \C^3 \mid  x^2+y^2+z^2 +xyz =px+qy+rz+s \right\},
    \end{align*}
where $$p=ab+cd, \quad q=bc+ad, \quad r=ac+bd, \quad s=4-a^2-b^2-c^2-d^2-abcd.$$

The mapping class group of $S$, $\MCG:=\pi_0({\rm Homeo}(S))$ acts on ${\mathcal X}_{\Ut}$, see \cite{gol_the}.
For concreteness, we adopt the convention here that $\MCG$ consists of orientation-preserving homeomorphisms fixing the boundary, there will be no essential difference to the ensuing discussion. The mapping class group is generated by Dehn twists along the simple closed curves corresponding to $\a\b$, $\b\g$ and $\a \g$. The action of each Dehn twist can be read easily on the trace coordinates of ${\mathcal X}_{\Ut}$. For example the Dehn twist about the separating curve $\a\b$ is the map $\C^3 \rightarrow \C^3$ given by:
$$  \begin{pmatrix} x \\ y \\ z \end{pmatrix} \longmapsto
    \begin{pmatrix} x \\ q -  x (r- xy -z) - y \\ r - xy - z \end{pmatrix},$$
      where $p,q,r$ are defined as above. This corresponds to the action of $\theta_2\theta_3$ given in the Introduction, see also Remark \ref{Dehn}.

\subsection{The BQ-conditions}\label{def_BQ}

For a fixed $\Ut$, let $K = K(\Ut) > 0$ be a constant depending only on $\Ut$, that we will define later in Definition $\ref{def:K}$.  A $\Ut$-representation $\rho \co \Gamma \to \SLtwoC$ (or $\PSLtwoC$) is said to satisfy the {\it BQ-conditions} (Bowditch's Q-conditions) if
\begin{enumerate}[{(BQ1)}]
  \item ${\rm tr}\rho(g) \not\in [-2,2]$  for all $[g] \in \hat \Omega$; and
  \item $|{\rm tr}\rho(g)| \le K$ for only finitely many (possibly no) $[g]\in \hat \Omega$.
\end{enumerate}

We also call such a representation $\rho$ a {\it BQ-representation}, or \textit{Bowditch representation}, and the space of equivalence classes of such representations the \textit{Bowditch representation space}, denoted by $({\mathcal X}_{\Ut})_{Q}$.

Note that ${\rm tr}\rho(g_1)={\rm tr}\rho(g_2)$ if $[g_1]=[g_2]$ (since $g_1$ is conjugate to $g_2$ or its inverse by definition); so the conditions (BQ1) and (BQ2) make sense.

\subsection{The binary tree $\Sigma$}

Let $\Sigma$ be a countably infinite simplicial tree properly embedded in the plane all of whose vertices have degree $3$. As an example, we can consider, as $\Sigma$, the binary tree dual to the Farey triangulation $\F$ of the hyperbolic plane $\HH$(also called an infinite trivalent tree). See \cite{tan_gen} for the definition of $\F$.

\subsection{Complementary regions}

A {\it complementary region} of $\Sigma$ is the closure of a connected component of the complement.

We denote by $\Omega=\Omega(\Sigma)$ the set of complementary regions of $\Sigma$. Similarly, we use $V(\Sigma)$, $E(\Sigma)$ for the set of vertices and edges of $\Sigma$ respectively.

\begin{figure}
[hbt] \centering
\includegraphics[height=2 cm]{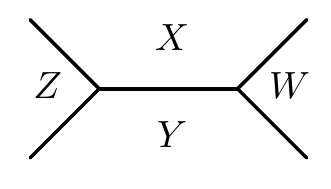}
\caption{The edge $e\leftrightarrow (X,Y;Z,W)$.}
\label{fig:edge oriented}
\end{figure}

We use the letters $X,Y,Z,W, \ldots$ to denote the elements of $\Omega$. For $e \in E(\Sigma)$, we also use the notation $e\leftrightarrow (X,Y;Z,W)$ to indicate that $e=X \cap Y$ and $e\cap Z$ and $e \cap W$ are the endpoints of $e$; see Figure \ref{fig:edge oriented}.

\subsection{A tri-coloring of the tree}

We choose a coloring of the regions and edges, namely a map $\mathcal{C} \co \Omega(\Sigma) \cup E(\Sigma) \to \{ 1, 2, 3 \}$ such that for any edge $e\leftrightarrow (X,Y;Z,W)$ we have $\mathcal{C}(e)=\mathcal{C}(Z) = \mathcal{C}(W)$ and such that $\mathcal{C} (e)$, $\mathcal{C}(X)$ , $\mathcal{C}(Y)$ are all different. The coloring is completely determined by a coloring of the three regions around any specific vertex, and hence is unique up to a permutation of the set $\{1, 2, 3\}$. We denote by $\Omega_i (\Sigma)$ the set of complementary regions with color $i$, and by $E_i (\Sigma)$ the set of edges with color $i$.

As a convention, in the following, when $X,Y,Z$ are complementary regions around a vertex, we will have $X \in \Omega_1 (\Sigma)$, $Y \in \Omega_2 (\Sigma)$ and $Z \in \Omega_3 (\Sigma)$.

\subsection{$\Um$--Markoff triples}

For a complex quadruple $\Um = (p,q,r,s) \in \C^4$, a $\Um$-{\it Markoff triple} is an ordered triple $(x,y,z)$ of complex numbers satisfying the $\Um$--Markoff equation:
\begin{eqnarray}\label{eqn:vertex}
x^2+y^2+z^2+xyz=px+qy+rz+s.
\end{eqnarray}
Note that, if $(x,y,z)$ is a $\mu$--Markoff triple in the sense of Tan-Zhang-Wong, then $(-x,-y,-z)$ is a $(0,0,0,\mu)$--Markoff triple in our sense.

It is easily verified that, if $(x,y,z)$ is a $\Um$--Markoff triple, so are the triples
\begin{equation}\label{eqn:elemoper}
(x,y,r-xy-z), \hspace{0.3cm} (x,q-xz-y,z) \mbox{   and    } (p-yz-x,y,z).
\end{equation}
It is important to note that permutations triples are not $\Um$--Markoff triples, contrary to the situation with $\mu$--Markoff triples.

\subsection{Relation with $\Ut$-representations}

Let $\mathrm{GT}\co \C^4 \longrightarrow \C^4$ be the map defined by:

\begin{align*}
    \left[
        \begin{array}{c}
            a \\ b \\ c \\ d
        \end{array}
    \right]
                        & \longmapsto
    \left[
        \begin{array}{c}
            ab+cd \\ ad+bc \\ ac+bd \\ 4-a^2-b^2-c^2-d^2 -abcd
        \end{array}
    \right] =   \left[
        \begin{array}{c}
            p \\ q \\ r \\ s
        \end{array}
    \right].
\end{align*}

This map is defined and studied by Goldman and Toledo in \cite{gol_aff}, where they show, among many other results, that the map is onto and proper. (Note that Goldman and Toledo denote this map $\Phi$.)

\begin{Remark}\label{corresp}
  Given $\Ut \in \C^4$, a representation $\rho$ is in $\mathcal{X}_{\Ut}$ if and only if $(\tr (\rho (\a \b )), \tr (\rho (\b \g )), \tr (\rho (\a \g )))$ is a $\Um$--Markoff triple with $\Um = \mathrm{GT} (\Ut)$.
\end{Remark}

The elementary operations defined in (\ref{eqn:elemoper}) are intimately related with the action of the mapping class group on the character variety, as we will see later.

\subsection{$\Um$--Markoff map}

A $\Um$-{\it Markoff map} is a function $\phi \co \Omega \to \C$ such that
\begin{itemize}
\item[(i)] for every vertex $v \in V(\Sigma)$, the triple $(\phi(X), \phi(Y), \phi(Z))$ is a $\Um$--Markoff triple, where $X,Y,Z \in \Omega$ are the three regions meeting $v$ such that $X \in \Omega_1$, $Y \in \Omega_2$ and $Z \in \Omega_3$;
\item[(ii)] For any $i \in \{1,2,3\}$ and for every edge $e \in E_i(\Sigma)$ we have:
\begin{itemize}
\item If $i=1$ and $e \leftrightarrow (Y,Z;X,X')$, then
\begin{eqnarray}\label{eqn:edge1}
x+x'=p-yz,
\end{eqnarray}
\item If $i=2$ and $e \leftrightarrow (X,Z;Y,Y')$, then
\begin{eqnarray}\label{eqn:edge2}
y+y'=q-xz,
\end{eqnarray}
\item If $i=3$ and $e \leftrightarrow (X,Y;Z,Z')$, then
\begin{eqnarray}\label{eqn:edge3}
z+z'=r-xy,
\end{eqnarray}
\end{itemize}
where $x=\phi(X), y=\phi(Y)$, $z=\phi(Z)$, $x'=\phi(X'), y'=\phi(Y')$ and $z'=\phi(Z')$\end{itemize}

We shall use ${\bf \Phi}_{\Um}$ to denote the set of all $\Um$--Markoff maps and lower case letters to denote the $\phi$ values of the regions. For example, we have $\phi(X)=x, ~\phi(Y)=y, ~\phi(Z)=z$.

\begin{Remark}\label{phi_ut}
  There exists a bijective correspondence between $\Um$--Markoff maps and $\Um$--Markoff triples. Hence, using Remark \ref{corresp}, there exists a bijective correspondence between the set ${\bf \Phi}_{\Um}$ of $\Um$--Markoff maps and the $\Ut$--relative character variety $\mathcal{X}_{\Ut}$, where $\Um = \mathrm{GT} (\Ut)$.
\end{Remark}

In fact, as in the case of Markoff maps and $\mu$--Markoff maps, if the edge relations \eqref{eqn:edge1}, \eqref{eqn:edge2} and \eqref{eqn:edge3} are satisfied along all edges, then it suffices that the vertex relation \eqref{eqn:vertex} is satisfied at a single vertex. So one may establish a bijective correspondence between $\Um$--Markoff maps and $\Um$--Markoff triples, by fixing three regions $X, Y, Z$  which meet at some vertex $v_0$. This process may be inverted by constructing a tree of $\Um$--Markoff triples as Bowditch did in \cite{bow_mar} for Markoff triples and as Tan, Wong and Zhang did in \cite{tan_gen} for the $\mu$--Markoff triples: given a triple $(x,y,z)$, set $\phi(X)=x, \phi(Y)=y, \phi(Z)=z$, and extend over $\Omega$ as dictated by the edge relations. In this way one obtains an identification of ${\bf \Phi}_{\Um}$ with the algebraic variety in $\C^3$ given by the $\Um$--Markoff equation. In particular, ${\bf \Phi}_{\Um}$ gets an induced topology as a subset of $\C^3$.

\subsection{The subsets $\Omega_{\phi}(k) \subset \Omega$}

Given $\phi \in {\bf \Phi}_{\Um}$ and $k \ge 0$, the set ${\Omega}_{\phi} (k) \subseteq \Omega$ is defined by
\begin{eqnarray*}
{\Omega}_{\phi}(k)=\{ X \in \Omega \mid |\phi(X)| \le k \}.
\end{eqnarray*}
These sets will be crucial in the proof of our main results.

We can now state the BQ-conditions in terms of Markoff maps.
\begin{Definition}
For a fixed $\Um$, let $L = L(\Um) > 0$ be a constant depending only on $\Um$, that we will define later in Definition $\ref{def:K}$.  A $\Um$-Markoff map $\phi \in {\bf \Phi}_{\Um}$ is said to satisfy the {\it BQ-conditions}  if
\begin{enumerate}[{(BQ1)}]
  \item $\phi^{-1}  ([-2,2]) = \emptyset$; and
  \item $\Omega_{\phi} (L)$ is finite.
\end{enumerate}
\end{Definition}
We denote by $({\bf \Phi}_{\Um} )_Q$ the set of all  $\Um$-Markoff maps which satisfy the BQ-conditions, and call the set of all such maps the Bowditch $\Um$-Markoff maps.

We will see that it is sufficient to give a large enough constant $L$ to define the BQ-conditions. Indeed, taking a bigger constant $L' > L(\Um)$ in the definition will give rise to the same subset of Markoff maps.

\section{Estimates on Markoff maps}\label{s:markoff}

The main aim of this section is to prove the following result, which is a generalisation of Theorem 3.1 of Tan--Wong--Zhang \cite{tan_gen}.

\begin{Theorem}\label{thm:omega} (Quasi-convexity)
Let $\Um \in \C^4$. Then:
\begin{enumerate}
\item There exists a constant $l = l ( \Um ) > 0$ such that $\forall \phi \in {\bf \Phi}_{\Um}$, we have that $\Omega_\phi (l)$ is non-empty.
\item There exists a constant $\alpha = \alpha ( \Um ) \geq 0$ such that $\forall \phi \in {\bf \Phi}_{\Um}$ and $\forall k \geq 2+\alpha$, the set $\Omega_\phi (k)$ is connected.
\end{enumerate}
\end{Theorem}

For doing that, we need to do lots of estimates. In order to shorten a bit the formulae which will appear, we introduce the following notation.

\begin{Notation}\label{alpha}
  Given a $\Um$--Markoff map $\phi$, let $$\alpha  = \alpha (\Um) = \frac{\max \{|p| , |q| , |r|\}}{2}.$$
\end{Notation}

For the rest of this section, let us fix $\Um \in \C^4$ and a $\Um$--Markoff map $\phi$.

\subsection{Arrows assigned by a $\Um$--Markoff map}

As Bowditch \cite{bow_mar} and Tan, Wong and Zhang \cite{tan_gen} did, we may use $\phi \in {\bf \Phi}_{\Um}$ to assign to each undirected edge, $e$, a particular directed edge, $\mathrm{vect}_\phi (e)$, with underlying edge $e$. Suppose $e \leftrightarrow (X,Y;Z,W)$. If $|z|>|w|$, then the arrow on $e$ points towards $W$; in other words, $\mathrm{vect}_\phi (e) = (X,Y;Z \rightarrow W)$. If $|z|<|w|$, we put an arrow on $e$ pointing towards $Z$, that is, $\mathrm{vect}_\phi (e) = (X,Y;W \rightarrow Z)$. If it happens that $|z|=|w|$ then we choose $\mathrm{vect}_\phi(e)$ arbitrarily. Let $\vec E(\Sigma)$ be the set of oriented edges.

A vertex with all three arrows pointing towards it is called a {\em sink}, one where two arrows point towards it and one away is called a {\em merge}, and vertex with two (respectively three) arrows pointing away from it is called a {\em fork} (respectively \emph{source}).

\begin{Lemma}\label{lem:fork} (Fork lemma)
Suppose $(X,Y,Z) \in \Omega_1 \times \Omega_2 \times \Omega_3$ meet at a vertex $v \in V (\Sigma)$, and the arrows on the edges $X \cap Y$ and $X\cap Z$ both point away from $v$. Then at least one of the following is true:
    \begin{itemize}
        \item $|x| \leq 2+ \frac{|q|+ |r|}{4}$;
        \item $|y|<2$;
        \item $|z| <2$.
    \end{itemize}
\end{Lemma}

We have similar results if the edges $Y \cap Z$ and $Y \cap X$ (or $Z\cap X$ and $Z \cap Y$) both point away from $v$.

\begin{proof}
Let $e_2$ and $e_3$ be the arrows pointing away from $v$. Let $Y'$ and $Z'$ be the regions such that $e_2 = (X,Z; Y \rightarrow Y')$ and $e_3 = (X , Y ; Z \rightarrow Z')$. The edge relation gives $|xz| = |y+y'-q| \leq |y|+|y'|+|q| \leq 2 |y|+|q|$. Similarly $|xy| \leq 2 |z|+  |r|$. Adding both inequalities, one obtains $$ |x| (|y|+|z|) \leq 2 (|y| + |z|) + (|q|+ |r|).$$ If $|y|\geq 2$ and $|z|\geq 2$, then we have $$|x| \leq 2 +  \frac{|q|+|r|}{|y|+|z|} \leq 2 + \frac{|q| + |r|}{4}.$$
\end{proof}

A weaker statement will be sufficient for most of the paper.
\begin{Corollary}
Suppose $(X,Y,Z) \in \Omega_1 \times \Omega_2 \times \Omega_3$ meet at a vertex $v \in V (\Sigma)$, and two arrows  point away from $v$, that is, $v$ is a fork or a source. Then
$$\min \{|x| , |y| , |z|\} \leq 2 + \alpha.$$
\end{Corollary}

\begin{Lemma}\label{lem:sink}
  There is a constant $m(\Um) \in \R_{>0}$ such that if three regions $X, Y, Z$ meet at a sink, then
$$ \min \{ |x| , |y| , |z| \} \leq m (\Um).$$
\end{Lemma}

\begin{proof}
We show that, if $|x|, |y|, |z|$ are all sufficiently large, then the vertex $v$ cannot be a sink. We may assume $x,y,z \neq 0$ and $\max\{|p|,|q|,|r|\}\leq |x| \leq |y| \leq |z|$. We can rewrite (\ref{eqn:vertex}) as: $$\frac{z}{xy}+ \frac{y}{xz} + \frac{x}{yz} + 1 = \frac{p}{yz} + \frac{q}{xz} + \frac{r}{xy}+ \frac{s}{xyz}.$$
There exists $K >0$ such that if $K <  |x| \leq |y| \leq |z|$, we have $|\frac{y}{xz}|, |\frac{x}{yz}|,| \frac{p}{yz}|, |\frac{q}{xz}|, |\frac{r}{xy}|, | \frac{s}{xyz} |$ are all smaller than $\frac{1}{12}$. It follows that
$$\left| \frac{z}{xy} +1 \right| < \frac{1}{2} \mbox{    and so    }  \left| \frac{z}{xy} \right| > \frac{1}{2}.$$
On the other hand, we have
$$\left| \frac{z}{xy} +1 - \frac{r}{xy}\right| < \frac{1}{2}.$$
So we infer that $|z| > |z+xy-r|$. Hence the arrow on the edge $X \cap Y$ is directed away from $v$ which proves that $v$ is not a sink.
\end{proof}

\subsection{Neighbors around a region}\label{sec:neighbors}

For each $X \in \Omega_1$, its boundary $\partial X$ is a bi-infinite path consisting of a sequence of edges of the form $X \cap Y_n$ alternating with $X\cap Z_n$, where $(Y_n)_{n \in \Z}$ and $(Z_n)_{n \in \Z}$ are bi-infinite sequences of complementary regions in $\Omega_2$ and $\Omega_3$. The edge relations \eqref{eqn:edge2}, \eqref{eqn:edge3} are such that
\begin{align*}
y_{n+1}     & = q - x z_n - y_n; \\
z_{n+1}     & = r - x y_{n+1} - z_n \\
            & = r - x q + (x^2-1)z_n + x y_n.
\end{align*}

\begin{Remark}\label{Dehn}
The map $(x, y_n , z_n) \longrightarrow (x , y_{n+1} , z_{n+1})$ is exactly the map defined by the Dehn twist along the curve $\a \b$ defined in the introduction.
\end{Remark}

We can reformulate these equations in terms of matrices:
\begin{equation}\label{eqn:sequence}
            \begin{pmatrix} y_{n+1} \\ z_{n+1} \end{pmatrix}
=           \begin{pmatrix} -1 & -x\\x & x^2-1  \end{pmatrix}
\cdot   \begin{pmatrix} y_{n} \\ z_{n} \end{pmatrix}
+           \begin{pmatrix} q \\ r-qx \end{pmatrix}.
\end{equation}

If $x \neq \pm 2$, this can be rewritten as
\begin{equation}\label{eqn:sequence2}
            \begin{pmatrix} y_{n+1} \\ z_{n+1} \end{pmatrix}
=           \begin{pmatrix} \mathfrak{y} (x) \\ \mathfrak{z} (x) \end{pmatrix}
+           \begin{pmatrix} -1 & -x\\x & x^2-1  \end{pmatrix}
\cdot
\left[  \begin{pmatrix} y_{n} \\ z_{n}\end{pmatrix}
-           \begin{pmatrix} \mathfrak{y}(x) \\ \mathfrak{z}(x) \end{pmatrix} \right],
\end{equation}
with
$$\mathfrak{y} (x) = \frac{1}{4- x^2} (2q - xr) , \hspace{1cm} \mathfrak{z}(x) = \frac{1}{4-x^2} (2r-xq).$$

{\bf Note:} $\mathfrak{y} (x)$ and $\mathfrak{z} (x)$ are the coordinates of the center of the conic in coordinates $(y,z)$ defined by the vertex relation.

The matrix $M:=\begin{pmatrix} -1 & -x\\x & x^2-1  \end{pmatrix} $ has determinant one. Hence its eigenvalues $\lambda$ and $\lambda^{-1}$ are such that $\lambda + \lambda^{-1} = \tr M = x^2 - 2$. Explicitly, if $\d$ is a square root of $x^2 -4$ in $\C$, then the eigenvalues are given by:
$$\lambda = \frac{x^2-2+x\d }{2} , \hspace{1cm} \lambda^{-1} = \frac{x^2 -2 - x \d }{2} $$
The matrix $M$ is
    \begin{itemize}
        \item elliptic, if $\Tr M \in [-2, 2) \iff x \in (-2, 2)$;
        \item parabolic, if $\Tr M = 2 \iff x = \pm 2$;
        \item loxodromic, if $\Tr M \notin [-2, 2] \iff x \notin [-2 , 2]$.
    \end{itemize}

\medskip

\noindent {\bf Case 1:} $M$ is elliptic, that is $x \in (-2 , 2)$.
Let $\theta \in (0 , \pi )$ such that $\theta = \arccos \left( \dfrac{x^2-2}{2} \right)$. In this case, there exists an invertible matrix $P$ such that
$$ M =  P \cdot \begin{pmatrix} e^{i\theta} & 0 \\ 0 & e^{i\theta} \end{pmatrix} \cdot  P^{-1}.$$

So, the sequences $(y_n)$, $(z_n)$ are given by:
\begin{equation}\label{eqn:sequenceelliptic}
        \begin{pmatrix} y_n \\ z_n \end{pmatrix} = \begin{pmatrix} \mathfrak{y}(x) \\ \mathfrak{z}(x) \end{pmatrix}+P \cdot  \begin{pmatrix} e^{i n \theta} & 0 \\ 0 & e^{i n \theta} \end{pmatrix} \cdot  P^{-1} \cdot \left[ \begin{pmatrix} y_0 \\ z_0 \end{pmatrix} - \begin{pmatrix} \mathfrak{y}(x) \\ \mathfrak{z}(x) \end{pmatrix} \right] .
\end{equation}
Hence, the sequences $(y_n)$ and $(z_n)$ are bounded.

\medskip

\noindent {\bf Case 2:} $M$ is parabolic, that is $x \in \{ -2 , 2 \}$.
When $x=2$, we get the exact formulae:
    \begin{align*}
        y_{n}   & = n^2 (q-r) + nr - (2n-1)(y_0 + z_0); \\
        z_{n}   & = n^2 (r-q) - nq + 2n (y_0 + z_0).
    \end{align*}
Similarly, when $x=-2$, we get the formulae:
    \begin{align*}
        y_{n}   & = n^2 (q+r) - nr + (2n-1)(z_0 - y_0); \\
        z_{n}   & = n^2 (q+r) + nq + 2n (z_0 - y_0).
    \end{align*}

\medskip

\noindent {\bf Case 3:} $M$ is loxodromic, that is $x \notin [-2 , 2]$.
In this case $x^2 \neq 4$ and so $I_2 - M$ is invertible. It follows that both $x$ and $\delta$ are non-zero. We have the following formula:
\begin{equation}\label{eqn:arith-geom}
            \begin{pmatrix} y_{n} \\ z_{n} \end{pmatrix}
=           \begin{pmatrix} \mathfrak{y}(x) \\ \mathfrak{z}(x) \end{pmatrix}
+ M^n \left[
            \begin{pmatrix} y_{0} \\ z_{0} \end{pmatrix}
-           \begin{pmatrix} \mathfrak{y}(x)  \\ \mathfrak{z}(x) \end{pmatrix}
\right].
\end{equation}

Calculations give
$$M^n = \frac{1}{\sqrt{x^2-4}}
        \begin{pmatrix}
                 \frac{\lambda+1}{x}(\lambda^{-n}-\lambda^{n-1}) & \lambda^{-n}-\lambda^n \\
                 \lambda^n - \lambda^{-n} & \frac{\lambda+1}{x}(\lambda^n-\lambda^{-n-1})
        \end{pmatrix}.$$

Using the definition of $\lambda$, we can see that $\Lambda = \frac{\lambda +1}{x} = \frac{x+\delta}2 $ is a square root of $\lambda$.  And similarly $\Lambda^{-1} = \frac{\lambda^{-1} +1}{x} = \frac{x-\delta}2$ is a square root of $\lambda^{-1}$.  Moreover we have $\delta = \Lambda - \Lambda^{-1}$. So we obtain the following closed formulae for $y_n$ and $z_n$:
    \begin{align*}
        y_{n}   & = \frac{1}{\Lambda - \Lambda^{-1}}
                \left( - \Lambda^{2n} \left(\Lambda^{-1} (y_0 - \mathfrak{y}) + (z_0 - \mathfrak{z}) \right) + \Lambda^{-2n} \left( \Lambda (y_0 - \mathfrak{y}) + (z_0 - \mathfrak{z})\right) \right) + \mathfrak{y}; \\
        z_{n}   & = \frac{1}{\Lambda - \Lambda^{-1}}
                \left( \Lambda^{2n+1} \left( \Lambda^{-1} (y_0 - \mathfrak{y}) + (z_0 - \mathfrak{z}) \right) - \lambda^{-2n-1} \left( \Lambda (y_0 - \mathfrak{y}) + (z_0 - \mathfrak{z})\right) \right) + \mathfrak{z}.
    \end{align*}

This means we can express these formulae as:
\begin{align*}
    y_n         &= A \Lambda^{2n} + B\Lambda^{-2n}+\mathfrak{y}; \\
    z_n         &= - (A \Lambda^{2n+1} +  B \Lambda^{-2n-1}) + \mathfrak{z},
\end{align*}

with

\begin{align*}
    A =     & \frac{-1}{\Lambda - \Lambda^{-1}} \left( \Lambda^{-1} (y_0 - \mathfrak{y}) + (z_0 - \mathfrak{z})\right); \\
    B=  & \frac{1}{\Lambda - \Lambda^{-1}} \left( \Lambda (y_0 - \mathfrak{y}) +   (z_0 - \mathfrak{z})\right).
\end{align*}

Hence, we see that both sequences have the same behaviour. We have $|\Lambda| = 1 \iff | \lambda| = 1 \iff x \in [-2 , 2]$. Hence, when $x \notin [-2 , 2]$ and $A, B$ are both non-zero, the sequences $|y_n|$ and $|z_n|$ grow exponentially in $n$ and $-n$. To determine when at least one of $A,B=0$, we have the following identity concerning the product $AB$:
    \begin{align*}
        AB  & = \frac{-1}{\delta^2} \left( \frac{x+\delta}{2}(y_0 - \mathfrak{y}) +  (z_0 - \mathfrak{z}) \right) \left( \frac{x-\delta}{2}(y_0 - \mathfrak{y}) +  (z_0 - \mathfrak{z}) \right) \\
                & = \frac{1}{4-x^2}     \left( (y_0 - \mathfrak{y})^2 +  \left( \frac{x+\delta}{2} + \frac{x-\delta}{2} \right) (y_0 - \mathfrak{y}) (z_0 - \mathfrak{z}) + (z_0 - \mathfrak{z})^2 \right) \\
                & = \frac{1}{4-x^2}     \left( y_0^2 + x y_0 z_0 + z_0^2 - y_0(2\mathfrak{y}+x\mathfrak{z}) - z_0 (2\mathfrak{z} + x\mathfrak{y} ) + \mathfrak{y}^2+\mathfrak{z}^2+x\mathfrak{y}\mathfrak{z} \right) \\
                & = \frac{1}{4-x^2}     \left( y_0^2 + x y_0 z_0 + z_0^2 - q y_0 - r z_0 + \frac{q^2+r^2-xrq}{4-x^2} \right)\\
                & = \frac{1}{4-x^2}     \left( px+s-x^2 + \frac{q^2+r^2-xrq}{4-x^2} \right).
    \end{align*}

\begin{Remark}
The vertex relationship is $$y_n^2+xy_nz_n+z_n^2-qy_n-rz_n+x^2-px-s = 0,$$ so it is a quadric and its type depends on the sign of the determinant $D = \mathrm{det}\left[ \begin{array}{cc} 1 & \frac{x}{2}\\\frac{x}{2} & 1  \end{array}\right] = 1 - \frac{x^2}{4}.$

For detecting when the conic is degenerate we need the following determinant:
\begin{align*}
  \Delta &= \mathrm{det}\left[ \begin{array}{ccc} 1 & \frac{x}{2} & -\frac{q}{2}\\ \frac{x}{2} & 1 & -\frac{r}{2} \\ -\frac{q}{2} & -\frac{r}{2} & (x^2-px-s) \end{array}\right] \\
  &= \frac{-x^4+px^3+(4+s)x^2+2(rq-2p)x-(q^2+r^2+4s)}{4}.
\end{align*}

Note that $\Delta = AB(x^2-4)^2$. We infer that the conic is degenerate when $x$ is a solution of the equation AB=0. Following Benedetto and Goldman's notation \cite{ben_the}, we can rewrite $\Delta$ as $\Delta = -\frac{\kappa_{a,b}(x)\kappa_{c,d}(x)}{4},$ where $(a,b,c,d)$ are such that $\mathrm{GT} (a,b,c,d) = (p,q,r,s)$ and $\kappa_{a,b}(x) = x^2-abx+a^2+b^2-4$.
\end{Remark}

Let $u , v \in \C$ and denote by $x_{u,v}^{\pm}$ the two solutions of the equation $\kappa_{u,v} (x) = 0$. Using this notation the solutions of the equation $AB = 0$ are $x_{a,b}^\pm , x_{c,d}^\pm $.

We note that $AB$ does not depend on $y_0$ and $z_0$. The coefficients $A$ and $B$ are both non-zero unless $x \in \{ x_{a,b}^\pm , x_{c,d}^\pm \}$. Note also that, in the case $p=q=r=0$, we recover the formula of Tan-Wong-Zhang, namely $AB = (x^2-s)/(x^2-4)$. We can now conclude the discussion with the following lemma.

\begin{Lemma}\label{lem:neighbors}
Suppose that $X \in \Omega_1$ has neighboring regions $Y_n$ and $Z_n$, $n \in \mathbb{Z}$. Then:
\begin{enumerate}
    \item If $x \in (-2,2)$, then $|y_n|$ and $|z_n|$ remain bounded.
    \item If $x = \pm 2$, then $|y_n|$ and $|z_n|$ grow at most quadratically.
    \item If $x \not\in[-2,2]$ and $x$ is not a solution of $AB=0$, then $|y_n|$ and $|z_n|$ grow exponentially as $n \rightarrow + \infty$ and as $n \rightarrow - \infty$.
    \item If $x \not\in[-2,2]$ and $x$ is a solution of $AB=0$, then:
        \begin{itemize}
            \item $\ds\lim_{n\rightarrow -\infty} y_n = \mathfrak{y}$, $\ds\lim_{n\rightarrow -\infty} z_n = \mathfrak{z}$ and $ |y_n|$ and $|z_n|$ grows exponentially as $n \rightarrow + \infty$;or
            \item $\ds\lim_{n\rightarrow +\infty} y_n = \mathfrak{y}$, $\ds\lim_{n\rightarrow +\infty} z_n = \mathfrak{z}$ and $ |y_n|$ and $|z_n|$ grows exponentially as $n \rightarrow - \infty$; or
            \item $y_n = \mathfrak{y}$ and $z_n = \mathfrak{z}$ for all $n \in \Z$.
        \end{itemize}
\end{enumerate}
\end{Lemma}

Given $\Um = \mathrm{GT} (a , b , c , d)$, let \label{S_mu}$$\mathcal{S}_{\Um} := \left\{ x_{i,j}^\pm | \{ i , j \} \subset \{ a, b , c , d \} \right\}.$$ Note that the case (4) of the previous Lemma can only happen for a Markoff map $\phi$ if one of the neighboring regions takes value in $\mathcal{S}_{\Um}$. In particular, the subcases of (4) correspond to the cases $A \neq 0$, $B \neq 0$ and $A=B=0$.

We can now give the definition of the constant $L(\Um)$ used to define the BQ-conditions. First define $M(\Um)$ as follows:
$$M(\Um) = \max \left\{ \left| \frac{2p_i - x p_j}{4-x^2} \right| , \,  x\in \mathcal{S}(\Um)\setminus \{\pm 2\} , p_i \neq p _j \in \{p,q,r\} \right\}.$$
So $M(\Um)$ is the maximum modulus of the coordinates of the center of the conic equation taken on the finite number of cases where the conic is degenerate.

\begin{Definition}\label{def:K}
Let $L(\Um) = \max \{ 2+\alpha , m(\Um) , M(\Um)+1\}$, and, similarly, we define $K(\Ut) =  L(\mathrm{GT} (\Ut))$.
\end{Definition}

\begin{Lemma}\label{lem:infiniteray}
Suppose $\beta$ is an infinite ray in $\Sigma$ consisting of a
sequence $(e_n)_{n\in \N}$ of edges of $\Sigma$ such that the arrow
on each $e_n$ assigned by $\phi$ is directed towards $e_{n+1}$. Then
the ray $\beta$ meets at least one region X with $| \phi (X) | \leq
2+ \alpha$ and an infinite number of regions with $|\phi(X)| \leq
L(\Um)$. Moreover, if $\phi^{-1}( \mathcal{S}_{\Um} ) = \emptyset$,
then $\beta$ meets an infinite number of regions with $|\phi(X)|\leq
2+\alpha$.
\end{Lemma}

\begin{proof}
\noindent $\bullet$ First we show that for any $\epsilon >0$, there
exists a region $X$ adjacent to $\beta$ such that $\phi (X) < 2 +
\alpha + \epsilon$.

Let $\eta = \frac{\epsilon}{2+\alpha}$ and let $(X_i), (Y_i)$ and $(Z_i)$ the sequences of neighboring regions along the infinite ray. The sequences $(|x_i|)$, $(|y_i|)$ and $(|z_i|)$ are decreasing and bounded below. So for $n$ large enough, we have two consecutive edges with region $X$ as a common face, i.e.  $\mathrm{vect}_\phi (e_n) = (X , Y; Z' \rightarrow Z)$ and $\mathrm{vect}_\phi (e_{n+1}) = (X , Z ; Y \rightarrow Y')$ and with $|z| \leq |z'| \leq |z|+\eta$. We can furthermore assume that $|y|>2$ and $|z|>2$. So we have easily
\begin{align*}
|xz| & = |y + y'-q|  \leq 2 |y| + |q|;  \\
|xy| & = |z+z'-r|  \leq 2|z'|+|r| \leq 2 |z| + |r| + 2 \eta .
\end{align*}
Multiplying both inequalities, we get
\begin{align*}
& &|x|^2 \cdot |yz| & \leq 4 |yz| + 2|y| \cdot |r| + 2 |z| \cdot |q| + |qr| + \eta ( 2|q| + 4 |y|),  \\
&\mbox{So} &  |x|^2 & \leq 4 + \frac{|r|}{|z|} + \frac{|q|}{|y|} + \frac{|qr|}{|yz|} + \eta \left(\frac{4}{|z|} + \frac{2|q|}{|yz|} \right),\\
& & & \leq \left( 2 + \alpha \right)^2 + \eta \left(2 + \alpha \right).
\end{align*}

Hence, we have that, for all $\epsilon >0$, there exists a region
$X$ adjacent to $\beta$ such that $|\phi(X)| < 2 + \alpha +
\epsilon$.

\medskip

\noindent $\bullet$ Next, suppose that for all neighboring regions of the infinite ray, the values are always greater than $(2+\alpha)$. By the above argument, for all $\epsilon> 0$ there exists a region $X$ such that $\phi (X) < 2+ \alpha + \epsilon$. So either $\phi (X) \leq 2 + \alpha$, or, for all $\epsilon$, there exists an edge $e_n$ in the ray such that $\mathrm{vect}_\phi (e_n) = (Y,Z ; X \rightarrow X')$ with $|x| + |x'|\leq 2 (2+\alpha)+ \epsilon$. Now the edge relation gives:

$$|x|+|x'| \geq |x + x' | \geq |yz - p| \geq |yz| - |p|. $$
On the other hand, we have:
\begin{align*}
|yz| - |p| & > (2+\alpha)^2 - |p| \\
& > 2 ( 2+ \alpha) + 2 \alpha - |p| + \alpha^2 \\
& > 2 ( 2+ \alpha) + \alpha^2.
\end{align*}

Hence we have $|x|+|x'| > 2 (2+\alpha) + \alpha^2$. Hence, if $\alpha \neq 0$, and $\epsilon$ is small enough, this gives a contradiction.

To conclude, we need to settle the case $\alpha = 0$, which is equivalent to $p=q=r=0$. In this case the triple $(-x,-y,-z)$ is a $s$--Markoff triple in the sense of Tan-Wong-Zhang, and the map is a Markoff map in their sense. So Lemma 3.11 of \cite{tan_gen} gives the argument in this case, that is, there exists a region with $|\phi(X) |< 2$.

\medskip

\noindent $\bullet$ For the second part of the statement, assume that there is only a finite number of regions meeting the ray such that $|\phi (X)| < 2+\alpha$. Then, as the ray is decreasing, it means that one of the sequences $(X_i)$, $(Y_j)$ or $(Z_k)$ that meet the infinite ray contains only a finite number of elements. Without loss of generality we can assume that it is the sequence $(X_i)$. Hence there exists $n_0$ such that for all $n \geq n_0$, the regions $Y_n$ and $Z_n$ are the neighbors of $X_{n_0}$. But we know that $(|y_j|)$ and $(|z_k|)$ are decreasing and bounded below, and from Lemma \ref{lem:neighbors}, the only possibility is that $X_{n_0} \in \phi^{-1}( \mathcal{S}_{\Um})$.

In addition, in the case where $X_{n_0} \in \phi^{-1}( \mathcal{S}_{\Um})$, the sequences $(y_n, z_n)$ converges to the center of the degenerate conic whose coordinates are strictly smaller than $L(\Um)$. Hence, there is an infinite number of $n$ such that $|y_n|$ and $|z_n|$ are smaller than $L(\Um)$.

\end{proof}

\begin{Example}\label{rmk:raywithfiniteintersection}
We give an explicit example of an infinite descending ray which intersects only a finite number of regions with values less than $2 + \alpha$.
With parameters $\Um = (0 ; 0 ; 1 ; 20)$, take a Markoff map $\phi$ such that around a vertex $v(X,Y_0,Z_0)$,
 $$x=\phi(X) = - \sqrt{3(4-\sqrt{7})} \approx -2.016 \geq -2-\alpha,$$   $$y_0=\phi(Y),  z_0=\phi(Z_0) \in \R, \quad y_0, z_0 >33$$ and $(x,y_0,z_0)$ satisfy the vertex equation (such $y_0$ and $z_0$ will always exist). Then the neighboring sequences $y_n$ and $z_n$ around $X$ (as $n \rightarrow \infty$) are both decreasing and converging towards a fixed point $\left(16+6\sqrt{7}, (8+3\sqrt{7}) \sqrt{3(4-\sqrt{7})}\right)$ which is around $(32 , 32)$. So this gives an infinite descending ray with only one neighboring region with value below $2+ \alpha$.
\end{Example}

We are now able to prove Theorem \ref{thm:omega}.

\begin{proof}[Proof of Theorem \ref{thm:omega}]
The proof follows the arguments of \cite{tan_gen}, using Lemmas \ref{lem:fork}, \ref{lem:sink} and \ref{lem:infiniteray}.
\begin{itemize}
  \item \textit{Proof of (i)}: Let $l = l(\Um) = \max \{m(\Um) , 2+ \alpha\}$. Suppose $\Omega_\phi(2+ \alpha)$ is empty, then Lemmas \ref{lem:fork} and \ref{lem:infiniteray} tell us that there exists a sink. Then Lemma \ref{lem:sink} states that around a sink, one of the region is such that $|\phi (X)| < m(\Um) $.
  \item \textit{Proof of (ii)}: Suppose the statement is false, then choose a minimal path on $\Sigma$ joining two different connected components. We will use induction on the number of edges of this path.

Suppose the path has one edge $e$. Without loss of generality we can assume $e \in E_3(\Sigma)$ and suppose $\vec e = (X,Y;Z \rightarrow W)$. Then we have:
$$ (2+\alpha)^2 < |xy| \leq |r| + |z| + |w| \leq 2(\alpha) + 2(2+\alpha),$$ so $\alpha^2 < 0$, which is absurd.

Suppose now that the path has more than one edge, then the two ends of the path are directed outwards and we derive a contradiction using Lemma \ref{lem:fork}, as discussed in detail in the proof of Theorem 1 (2) of Bowditch \cite{bow_mar}.
\end{itemize}
\end{proof}

\section{Domain of discontinuity}\label{s:domain}

The aim of this section is to prove Theorem \ref{thmA} from the Introduction. In particular, we will prove that the set $({\bf \Phi}_{\Um})_{Q}$ of maps satisfying the BQ-conditions (where the constant $K$ is taken to be $L(\Um)$, equivalently $K(\Ut)$ given by Definition \ref{def:K})
is an open subset of ${\bf \Phi}_{\Um}$ in Theorem \ref{thm:openproperty}, and that the mapping class group acts properly discontinuously on it in Theorem \ref{thm:properlydiscontinuously}. To do so, we will use the notion of Fibonacci growth, already used by Bowditch \cite{bow_mar} and Tan, Wong and Zhang \cite{tan_gen}.

\subsection{Fibonacci functions}

In this section we recall the definition, given in \cite{bow_mar}, of Fibonacci function $F_e$ associated to an edge $e \in E(\Sigma)$, and of upper or lower Fibonacci bound for a function $f\co\Omega \rightarrow [0, \infty)$.

Suppose $\vec{e} \in  \vec E(\Sigma)$, and set $\Omega^0$ to be $\{ X , Y \}$, where $e = X \cap Y$. Let $\Sigma^{\pm}$ be the two disjoint subtrees obtained when removing the interior of $e$ (such that $\Sigma^+$ is at the head of $\vec e$), and let  $\Omega^{0\pm}$ the set of regions whose boundaries  lie in $\Sigma^{\pm}$. We also denote $\Omega^{0\pm} = \Omega^0 \cup \Omega^{\pm}$.

First, we recall the notion of distance. Given $\vec e \in \vec E(\Sigma)$ with underlying edge $e$, we describe the function $d\co\Omega^{0-}(\vec e) \rightarrow \N$. For $X \in \Omega^{0-}(\vec e)$ we define $d(X)=d_{\vec e}(X)$ to be the number of edges in the shortest path joining the head of $\vec e$ to $X$. Given any $Z \in \Omega^{-}(\vec e)$, there are precisely two regions $X,Y \in \Omega^{0-}(\vec e)$ meeting $Z$ and satisfying $d(X)<d(Z)$ and $d(Y)<d(Z)$. Note that $X,Y,Z$ all meet in a vertex.

Now we can define the {\it Fibonacci function} $F_e\co \Omega \rightarrow {\mathbf N}$ with respect to an edge $e$ as follows. We orient $e$ arbitrarily as $\vec e$ and define $F_{\vec e}\co \Omega^{0-}(\vec e) \rightarrow {\mathbf N}$ by
$$F_{\vec e}(Z)=
\begin{cases}
1,  & \mbox{if } Z \in \Omega^{0}(e)\\
F_{\vec e}(X)+F_{\vec e}(Y), & \mbox{if } Z \in \Omega^{-}(\vec e),
\end{cases}$$
and $X,Y \in
\Omega^{0-}(\vec e)$ are the two regions described above: $X\cap Y \cap Z \neq \varnothing$ and $d(X)<d(Z), d(Y)<d(Z)$. Now we define $F_e$ by the following:
$$F_e(X)=
\begin{cases}
F_{\vec e}(X),  & X \in \Omega^{0-}(\vec e)\\
F_{-\vec e}(X), & \mbox{if } X \in \Omega^{+}(\vec e).
\end{cases}$$

The functions $F_e$ provide a means for measuring the growth rates of functions defined on subsets of $\Omega$. The following lemma can be easily proved by induction. Its corollary shows that the concept of upper and lower Fibonacci bound is independent of the edge $e$ used.

\begin{Lemma}\label{lem:B2.1.1}{\rm (Lemma 2.1.1 in \cite{bow_mar})}
Suppose $f\co \Omega^{0-}(\vec e) \rightarrow [0, \infty)$, where
$\Omega^{0}(e)=\{X_1,X_2\}$.
\begin{enumerate}[(i)]
  \item If $f$ satisfies $f(Z) \le f(X)+f(Y)+c$ for some fixed constant $c$ and arbitrary $X, Y, Z \in \Omega^{0-}(\vec e)$ meeting at a vertex and satisfying $d(X)<d(Z)$ and $d(Y)<d(Z)$, then $f(X) \le (M+c)F_e(X)-c$ for all $X \in \Omega^{0-}(\vec e)$, where $M=\max\{f(X_1), f(X_2)\}$.
  \item If $f$ satisfies $f(Z) \ge f(X)+f(Y)-c$ for some fixed constant $c$, where $0<c<m=\min\{f(X_1), f(X_2)\}$ and arbitrary $X,Y,Z$ as in part (i), then $f(X) \ge (m-c)F_e(X)+c$ for all $X \in \Omega^{0-}(\vec e)$.
\end{enumerate}
\end{Lemma}

\begin{Corollary}\label{cor:B2.1.2}{\rm (Corollary 2.1.2 in \cite{bow_mar})}
Suppose $f\co \Omega \rightarrow [0, \infty)$ satisfies an inequality of the form $f(Z) \le f(X)+f(Y)+c$ for some fixed constant $c$, whenever $X,Y,Z \in \Omega$ meet at a vertex. Then for any given edge $e \in E(\Sigma)$, there is a constant $K>0$, such that $f(X) \le KF_e(X)$ for all $X \in \Omega$.
\end{Corollary}

Now we can define what  it means for a function $f\co\Omega \rightarrow [0, \infty)$ to have an upper or lower Fibonacci bound.

\begin{Definition}
  Suppose $f\co \Omega \rightarrow [0, \infty)$, and $\Omega^{\prime} \subseteq \Omega$. We say that:
  \begin{itemize}
    \item $f$ has an {\it upper Fibonacci bound on} $\Omega^{\prime}$ if there is some constant $\kappa > 0$ such that $f(X) \le \kappa\, F_e(X)$ for all $X \in \Omega^{\prime}$;
    \item $f$ has a {\it lower Fibonacci bound on} $\Omega^{\prime}$ if there is some constant $\kappa > 0$ such that  $f(X) \ge \kappa\, F_e(X)$ for all but finitely many $X \in \Omega^{\prime}$;
    \item $f$ has {\it Fibonacci growth} on $\Omega^{\prime}$ if it has both upper and lower Fibonacci bounds on $\Omega^{\prime}$;
    \item $f$ has {\it Fibonacci growth} if $f$ has Fibonacci growth on {\it all} of $\Omega$.
  \end{itemize}
\end{Definition}

The following lemma tells us that, given an arbitrary $\Um$--Markoff map $\phi$, the function $\log^{+}|\phi|:= \max\{\log|\phi|, 0\}$  always has an upper Fibonacci bound on $\Omega$. Hence we will only need to consider criteria for it to have a lower Fibonacci bound on certain branches of the binary tree $\Sigma$ in order to determine if it has Fibonacci growth.

\begin{Lemma}\label{lem:upper_bound}
If $\phi \in {\bf \Phi}_{\Um}$, then $\log^{+}|\phi|$ has an upper Fibonacci bound on $\Omega$.
\end{Lemma}

\begin{proof}
We will follow the arguments of \cite{tan_gen}. By Corollary \ref{cor:B2.1.2} we only need to show that for an arbitrary $\Um$--Markoff map $(x,y,z)$ we have:
\begin{equation}\label{eqn:upper_bound}
  \begin{split}
  \log^{+}|z| & \le \log^{+}|x|  + \log^{+}|y| + \log 10 \\
  &   + \log^{+}|p|+\log^{+}|q|+ \log^{+}|r| + \log^{+}|s|.
  \end{split}
\end{equation}

If $|z| \le 2|x|$ or $|z|\le 2|y|$, then \eqref{eqn:upper_bound} holds already. So we suppose $|z| \ge 2|x|$ and $|z| \ge 2|y|$. Then, since $px+qy+rz+s - xyz = x^2 + y^2 + z^2$, we have:
\begin{eqnarray*}
|px| + |qy| + |rz| + |s| +|xyz| &\ge& |z|^2-|x|^2-|y|^2 \\
&=& |z|^2/2 + (|z|^2/4-|x|^2) + (|z|^2/4-|y|^2) \\
&\ge& |z|^2/2.
\end{eqnarray*}
Hence $|z|^2 \le 10 \max\{|px|, |qy|, |rz|, |s|, |xyz|\}$, that is, according to the value of $\max\{|px|, |qy|, |rz|, |s|, |xyz|\}$, we have, respectively:
\begin{enumerate}
  \item $|z|^2 \le 10 |px|$;
  \item $|z|^2 \le 10 |qy|$;
  \item $|z|^2 \le 10 |rz|$;
  \item $|z|^2 \le 10 |s|$; or
  \item $|z|^2 \le 10 |xyz|$.
\end{enumerate}
Thus, since we may assume $|z| \ge 1$, we have respectively:
\begin{enumerate}
  \item $|z| \le |z|^2 \le 10|px|$;
  \item $|z| \le |z|^2 \le 10|qy|$;
  \item $|z| \le 10|r|$;
  \item $|z| \le |z|^2 \le  10|s|$; or
  \item $|z| \le  10|xy|$.
\end{enumerate}
From this, Equation \eqref{eqn:upper_bound} follows easily.
\end{proof}

The lower Fibonacci bounds are more interesting since, as the following proposition shows, they give the convergence of certain series, see \cite{bow_mar}.

\begin{Proposition}\label{prop:B2.1.4}{\rm (Proposition 2.1.4 in \cite{bow_mar})}
If $f\co \Omega \rightarrow [0, \infty)$ has a lower Fibonacci bound, then ${\sum}_{X\in\Omega}f(X)^{-s}$ converges for all $s>2$ (after excluding a finite subset of $\Omega$ on which $f$ takes the value $0$).
\end{Proposition}

The following lemma and corollary hold with proofs similar to the ones given in \cite{bow_mar}, giving a criterion for $\log^{+}|\phi|$ to have a lower Fibonacci bound on certain branches of $\Sigma$. 

\begin{Lemma}\label{lem:B3.5}
  \begin{enumerate}[(i)]
        \item Suppose $\vec e \in \vec E(\Sigma)$ is such that $\mathrm{vect}_\phi(e)=\vec e$ and $\Omega^0(e) \cap \Omega_\phi(2+\alpha) = \emptyset$. Then $\Omega^{0-}(\vec e) \cap \Omega_\phi(2+\alpha) = \emptyset$ and the arrow on each edge of $\Sigma^{-}$ is directed towards $e$.

        \item Furthermore, $\log |\phi(X)| \ge (m - \log 2)F_e(X)$ for all $X \in \Omega^{0-}(\vec e)$, where $m=\min \{\log |\phi(X)| \mid X \in \Omega^0(e)\} > \log (2+\alpha)$.
    \end{enumerate}
\end{Lemma}

\begin{Corollary}\label{cor:B3.7}
  If $\Omega_\phi(2+\alpha)=\emptyset$, then there is a unique sink, and $\log^{+}|\phi|$ has a lower Fibonacci bound.
\end{Corollary}

\begin{proof}[Proof of Lemma \ref{lem:B3.5}]
   We will follow Bowditch's discussion \cite{bow_mar}. Suppose $\vec e \in \vec E(\Sigma)$ is such that $\alpha(e)=\vec e$. Since $\Omega_\phi (2+\alpha)$ is connected, we must have either $\Omega_\phi (2+\alpha) \subset \Omega^{+}(\vec e)$ or $\Omega_\phi (2+\alpha) \subset \Omega^{-}(\vec e)$ (possibly $\Omega_\phi (2+\alpha)= \emptyset$). If $\Omega_\phi (2+\alpha) \subset \Omega^{-}(\vec e)$ and $\Omega_\phi (2+\alpha) \neq \emptyset$, then $\alpha(e)=-\vec e$ (using an argument similar to the proof of Theorem \ref{thm:omega} (ii)). This proves the first part of (i). For proving that arrows are directed towards $e$ one should use Lemma \ref{lem:fork}.

  For the second part of the lemma, let $X, Y, Z \in \Omega^{0-}(\vec e)$ meeting at a vertex and satisfying $d(X)<d(Z)$ and $d(Y)<d(Z)$, as in Lemma \ref{lem:B2.1.1}. Without loss of generality suppose $e \in E_3(\Sigma)$. By (i) we know that the arrow on $X \cap Y$ points away from $Z$. We know also that $|r| \leq 2 \alpha \leq |z| \alpha$. So we have that $|xy| \le 2|z| + |r| \le (2+\alpha) |z|$. Thus $\log^{+} |\phi(Z)| \ge \log^{+}|\phi(X)| + \log^{+}|\phi(Y)| - \log (2+\alpha)$. We can now apply Lemma \ref{lem:B2.1.1} (ii).
\end{proof}

\begin{proof}[Proof of Corollary \ref{cor:B3.7}]
  Again, following Bowditch's ideas, we recall that the existence comes from Lemma \ref{lem:infiniteray}, the uniqueness from Lemma \ref{lem:fork}, while the upper Fibonacci bound comes from Lemma \ref{lem:upper_bound} and the lower one from Lemma \ref{lem:B3.5}.
\end{proof}

We are ready to complete the aim of this subsection which is to prove the following result: Recall the definition of $({\bf \Phi}_{\Um})_{Q}$ given in Section \ref{def_BQ}.

\begin{Theorem}\label{thm:B2}
Suppose $\phi \in ({\bf \Phi}_{\Um})_{Q}$, then ${\log}^{+}|\phi|$ has Fibonacci growth.
\end{Theorem}

In order to prove it, we expand the discussion of Lemma \ref{lem:B3.5} and we consider the case where $\Omega^{-}(\vec e) \cap \Omega_\phi(2)=\emptyset$ and exactly one of the two regions in $\Omega^0(e)$ has norm no greater than $2+ \alpha$.

\begin{Lemma}\label{lem:B3.8}
Suppose $\vec e \in \vec E(\Sigma)$ is such that $\mathrm{vect}_\phi(e)=\vec e$ and $\Omega^{0-}(\vec e) \cap \Omega_\phi(2+\alpha)=\{X_0\}$, where $X_0 \in \Omega^0(\vec e)$ with $x_0 \notin [-2,2]$. Then $\log^+|\phi|$ has a Fibonacci bound on $\Omega^{0-}(\vec e)$.
\end{Lemma}

\begin{proof}
  Our proof slightly modifies Bowditch's proof of Lemma 3.8 in \cite{bow_mar}. Let $(\vec e_n, \vec f_n)_{n=0}^{\infty}$ be the sequence of directed edges lying in the boundary of $X_0$ and in $\Omega^{0-}(\vec e)$ so that $\vec e_0 = \vec e$ and $\vec e_n$ is directed away from $\vec f_{n+1}$ and towards $\vec f_{n}$, and $\vec f_n$ is directed away from $\vec e_{n}$ and towards $\vec e_{n-1}$. For $n \ge 1$, let $v_n$ be the vertex incident on both $e_{n}$ and $f_{n}$ and $u_n$ be the vertex incident on both $e_{n}$ and $f_{n+1}$; let aso $\vec \varepsilon_n$ be the third edge (distinct from $e_{n}$ and $f_{n}$) incident on $v_n$ and directed towards $v_n$, and, similarly, let $\vec \xi_n$ be the third edge (distinct from $e_{n}$ and $f_{n+1}$) incident on $u_n$ and directed towards $u_n$. For $n \ge 0$, let $Y_n$ and $Z_n$ be the regions such that $Y_n \cap X_0 = e_n$ and $Z_n \cap X_0 = f_n$. Thus $\Omega^{0-}(\vec e) = \{X_0\} \cup \bigcup_{n=1}^{\infty}\left(\Omega^{0-}(\vec \varepsilon_n)\cup \Omega^{0-}(\vec \xi_n)\right)$. We recall that, in this context, Bowditch noticed that, using Lemma \ref{lem:B2.1.1}(ii), a map $f \co \Omega^{0-}(\vec e) \to [0, \infty)$ has a lower Fibonacci bound on $\Omega^{0-}(\vec e)$ if and only if there is some constant $k > 0$ such that for all $n \ge 1$ and for all $X \in \Omega^{0-}(\vec \varepsilon_n)$ we have $f(X) \ge knF_{\vec \varepsilon_n}(X)$ and, similarly, for all $X \in \Omega^{0-}(\vec \xi_n)$ we have $f(X) \ge knF_{\vec \xi_n}(X)$.

By Lemma \ref{lem:neighbors}, $|y_n|$ and $|z_n|$ grows exponentially as $n \rightarrow \infty$, and so $\log|y_n| \ge cn$ and $\log|z_n| \ge c'n$ for some constants $c, c'>0$. Hence we have, for all $n \ge 1$, that $\log^+|\phi(X)| \ge cnF_{\vec \varepsilon_n}(X)$, for all $X \in \Omega^{0-}(\vec \varepsilon_n)$, and that $\log^+|\phi(X)| \ge cnF_{\vec \xi_n}(X)$, for all $X \in \Omega^{0-}(\vec \xi_n)$. Thus, using the characterisation above, it follows that $\log^+|\phi|$ has a lower Fibonacci bound on $\Omega^{0-}(\vec e)$.

Since the upper Fibonacci bound was proved in Lemma \ref{lem:upper_bound}, the proof is done.
\end{proof}

We can now prove Theorem \ref{thm:B2}.

\begin{proof}
  The proof of Theorem \ref{thm:B2} is then the same as that of Theorem 2 in \cite{bow_mar}. We sketch it as follows. By Lemma \ref{lem:upper_bound}, we only need to show that $\log^+|\phi|$ has a lower Fibonacci bound on $\Omega$. If $\Omega_\phi(2+\alpha)$ has at most one element, the conclusion follows easily by Corollary \ref{cor:B3.7} and Lemma \ref{lem:B3.8}. Hence we can suppose $\Omega_\phi(2+\alpha)$ has at least two elements.

  Recall that $\Omega_\phi(2+\alpha) \subseteq \Omega$ is finite and $\bigcup\Omega_\phi(k)$ is connected for any $k  \ge 2+\alpha$, see Theorem \ref{thm:omega}. Let $T$ be the (finite) subtree of $\Sigma$ {\it spanned} by the set of edges $e$ such that $\Omega^0(e) \subseteq \Omega_\phi(2+\alpha)$. Let $C=C(T)$ be the circular set of directed edges given by $T$. Note that $\Omega_\phi (2+\alpha)=\bigcup_{\vec e \in C}\Omega^{0-}(\vec e)$. Hence it suffices to show that $\log^+|\phi|$ has a lower Fibonacci bound on $\Omega^{0-}(\vec e)$ for every $\vec e \in C$. Then the conclusion of Theorem \ref{thm:B2} follows by the following lemma, Lemma \ref{lem:B3.5} and Lemma \ref{lem:B3.8}.
\end{proof}

\begin{Lemma}
  For each $\vec e \in C$, we have $\vec e = \mathrm{vect}_\phi(e)$, $\Omega^-(\vec e) \cap \Omega_\phi(2+\alpha) = \emptyset$ and $\Omega^0(e) \cap \Omega_\phi(2+\alpha)$ has at most one element.
\end{Lemma}

\begin{proof}
  Let $\vec e = (X,Y;Z \rightarrow W) \in C(T)$. If one of $X$ and $Y$, say $X$, is in $\Omega_\phi(2+\alpha)$, then $Y, Z \notin \Omega_\phi(2+\alpha)$ and $W \in \Omega_\phi(2+\alpha)$ by the definition of $T$. Hence in this case $\vec e = \mathrm{vect}_\phi(e)$, $\Omega^-(\vec e) \cap \Omega_\phi(2+\alpha) = \emptyset$ and $\Omega^0(e) \cap \Omega_\phi(2+\alpha)$ has one element, $X$.

Now suppose neither $X$ nor $Y$ is in $\Omega_\phi(2+\alpha)$, then $W \in \Omega_\phi(2+\alpha)$ and $Z \notin \Omega_\phi(2+\alpha)$ since $\bigcup\Omega_\phi(2+\alpha)$ is connected. Thus in this case $\vec e = \mathrm{vect}_\phi(e)$, $\Omega^{0-}(\vec e) \cap \Omega_\phi(2+\alpha) = \emptyset$. This proves the lemma, completing the proof of the theorem.
\end{proof}

Assuming Theorem \ref{thm:B2}, the following result becomes an easy Corollary of Proposition \ref{prop:B2.1.4}:

\begin{Corollary}\label{cor:sum phi^-2 convergence}
If $\phi \in ({\bf \Phi}_{\Um})_{Q}$, then for any $ t>0$, the series $\sum_{X \in \Omega}|{\phi}(X)|^{t}$ converges absolutely.
\end{Corollary}

\subsection{Openness and properly discontinuous action}

In this section we will prove Theorem \ref{thmA}, using the definition of some attracting tree $T(t)$ for $t \geq 0$. In particular, first, we will define the attracting subtree $T(0)$, then we will use the existence of a certain map $H_{\Um}$, proved in Lemma \ref{HMu}, to define the attracting subtree $T(t)$.

\subsubsection{Attracting tree $T(0)$}

From Lemma \ref{lem:neighbors}, we see that, when $x = \phi (X) \notin [-2, 2] \cup \mathcal{S}_{\Um}$, the sequences $|y_n|$ and $|z_n|$ are monotonic for sufficiently large and sufficiently small $n$. (The set $\mathcal{S}_{\Um}$ was defined at page \pageref{S_mu}.) From this observation, we can prove the following result.

\begin{Lemma} If $X \in \Omega$ and $\phi (X) \notin [-2 , 2]\cup \mathcal{S}_{\Um}$, then there is a non-empty finite subarc $J(X) \subset \partial X$ with the property that, if $e$ is any edge in $\partial X$ not lying in $J(X)$, then the arrow on $e$ points towards $J(X)$.
Moreover, we can assume that $Y \cap X \subset J(X)$ for all $Y \in
\Omega_\phi (2+\alpha)$.
\end{Lemma}

When $x = \phi (X) \in [-2,2] \cup \mathcal{S}_{\Um}$, we shall set $J(X) = \partial X$.

Now, given $\phi \in {\bf \Phi}_{\Um}$ with $\Omega_\phi (2+\alpha) \neq \emptyset$ (where $\alpha$ was introduced in Notation \ref{alpha}), we define:
$$T(0) := \bigcup_{X\in \Omega_\phi (2+\alpha)} J(X).$$

If $\Omega_\phi(2+\alpha ) = \emptyset$, then, applying Corollary \ref{cor:B3.7}, we can take $T(0)$ to be the unique sink.

\begin{Lemma}\label{lem:subtree}
$T(0)$ is connected, and the arrow on each edge not in $T(0)$ points towards $T(0)$.
\end{Lemma}

We say that $T(0)$ is an \emph{attracting subtree}.

\begin{proof}
The proof is elementary from Lemma \ref{lem:fork}. The arrows on every edge on the circular boundary of a connected component $T'$ of $T(0)$ points towards $T'$. Hence if we suppose there exists an arc outside $T(0)$ with its two endpoints on the boundary of $T(0)$, there is a vertex $v$ with two edges that point away from $v$, and hence the vertex $v$ would belong to $\Omega_\phi (2+\alpha)$, which gives rise to a contradiction. The same applies if an arrow outside $T(0)$ does not points toward $T(0)$.
\end{proof}

\begin{Corollary}
If $\phi \in ({\bf \Phi}_{\Um})_{Q}$, then there is a finite attracting subtree $T(0)$.
\end{Corollary}

\begin{proof}
If $\Omega_\phi (2+\alpha) \neq \emptyset$, then the tree $T(0)$ given by Lemma \ref{lem:subtree} is finite. Indeed, as $\phi \in ({\bf \Phi}_{\Um})_{Q}$, there cannot be an infinite descending sequence (or bounded sequence) of regions  and hence $\phi^{-1} ([-2,2]\cup \mathcal{S}_{\Um} ) = \emptyset$.  If $\Omega_\phi (2+ \alpha ) = \emptyset$, then the result is true by definition.
\end{proof}

\subsubsection{Function $H_{\Um}$}

Recalling the definition of $\mathcal{S}_{\Um}$ given at page \pageref{S_mu}, we can state the following result.

\begin{Lemma}\label{HMu}
There exists a function $H_{\Um} \co \C \setminus ([-2 , 2] \cup \mathcal{S}_{\Um}) \rightarrow \R_{>0}$ so that for any $\phi \in {\bf \Phi}_{\Um}$ and $X \in \Omega$, if $(Y_n)$ and $(Z_n)$ are the bi-infinite sequence of regions meeting $X$, then there are integers $n_1 \leq n_2$ such that
$$|y_n| \leq H_{\Um} (x) \;\& \;|z_n| \leq H_{\Um} (x) \Leftrightarrow n_1 \leq n \leq n_2$$
and $|y_n|$ and $|z_n|$ are monotomically decreasing for $n \in (-\infty , n_1)$ and monotonically increasing for $n \in [n_2 , \infty)$. Moreover, $H_{\Um} (x) \geq 2 + \alpha$.
\end{Lemma}

\begin{proof}
From the formulae for $y_n$ and $z_n$ found at the end of the discussion of Case 3 in Section \ref{sec:neighbors}, it is an easy exercise.
\end{proof}

\begin{Remark}
A formula for $H_{\Um} (x)$ would be quite hard to write explicitly.
\end{Remark}

\subsubsection{Attracting tree $T(t)$}

For $x \in [-2 , 2] \cup \mathcal{S}_{\Um}$, we set $H_{\Um}(x) = H(x) = \infty$.

Now for $X \in \Omega$ with $x = \phi (X) \notin [-2 , 2] \cup \mathcal{S}_{\Um}$ and for $r \geq H(x)$, we set
$$J_r (X) = \bigcup \{ (X \cap Y_n) \cup (X\cap Z_n) |\,  |y_n|\leq r \mbox{ and } |z_n|\leq r \}.$$

$J_r (X)$ is a subarc of $\partial X$ with the property that, if $e$ is any edge on $\partial X \setminus J_r (X)$, then the arrow on $e$ points towards $J_r (X)$.

We now define an attracting subtree. Let
$$T(t) := \bigcup_{X \in \Omega_\phi (2+\alpha+t) } J_{H(x)+t}(X).$$

\begin{Lemma}
For all $t\geq 0$, the subtree $T(t)$ is connected and attracting. Moreover, if $t \geq m(\Um) - 2 - \alpha$, then $T(t) \neq \emptyset$.
\end{Lemma}

We can describe $T(t)$ directly in terms of its edges. Suppose $e = X \cap Y$ is an edge. Then
$$e \in T(t) \Leftrightarrow \left\{ \begin{array}{ll} |x| \leq 2 + \alpha + t & \mbox{ and } |y|\leq H(x) + t \\ \mbox{ or } & \\ |y| \leq 2 + \alpha + t & \mbox{ and } |x|\leq H(y) + t \end{array} \right. $$

We can now prove the following lemma.

\begin{Lemma}\label{lem:finitetree}
For all $t \geq 0$, $\phi \in ({\bf \Phi}_{\Um})_{Q}$ if and only if $T(t)$ is finite.
\end{Lemma}

\begin{proof}
Let $\phi \in ({\bf \Phi}_{\Um})_{Q}$, then for each $X \in \Omega_\phi (2+\alpha+t)$ and $t \geq 0$, the arc $J_{H(x)+t} (X)$ has a finite number of edges as $x \notin [-2 , 2] \cup \mathcal{S}_{\Um}$. On the other hand, the lower Fibonacci bound proves that $\Omega_\phi (2 + \alpha + t )$ is finite, hence $T(t)$ is finite.

Reciprocally, suppose $T(t)$ is finite. Then it is clear that $\phi^{-1} ([-2 , 2] \cup \mathcal{S}_{\Um}) = \emptyset$. Now, for all $X \in \Omega_\phi (2 + \alpha + t)$, the arc $J_{H(x)+t} (X)$ contains at least one edge. Hence we infer that $\Omega_\phi (2 + \alpha + t)$ is finite, so $\phi \in ({\bf \Phi}_{\Um})_{Q}$.
\end{proof}

\begin{Theorem}\label{thm:openproperty}
The set $({\bf \Phi}_{\Um})_{Q}$ of maps satisfying the BQ-conditions is an open subset of ${\bf \Phi}_{\Um}$.
\end{Theorem}

\begin{proof}
Let $\phi \in ({\bf \Phi}_{\Um})_{Q}$, and $t_1 > t_0 > m(\Um) - 2 - \alpha$. By Lemma \ref{lem:finitetree}, $T(t_1)$ is a finite subtree of $\Sigma$, so we may choose $t_2>t_1$ large enough so that $T(t_2)$ contains $T(t_1)$ in its interior, that is, it contains $T(t_1)$ together with all the edges of the circular set $C(T(t_1))$. Note that $T(t_2)$ is also a finite subtree of $\Sigma$. For any $\phi' \in {\bf \Phi}_{\Um}$, we write $T'(t)$ for $T_{\phi'} (t)$.

\textbf{Claim}: If $\phi'$ is sufficiently close to $\phi$, then $T'(t_1) \cap T(t_2) \subseteq T(t_1)$.

To prove the claim, choose an edge $e \in T(t_2) \setminus T(t_1)$ with neighboring region $e=X\cap Y$. Since $e \in T(t_2)$, we may assume $|\phi (X)| \leq 2 + t_2 + \alpha$ and $|\phi (Y)| \leq H(x) + t_2$. Then, as $e \notin T(t_1)$, we have  ($|\phi (X)| > 2 + t_1 + \alpha$ or $|\phi (Y)| > H(x) + t_1$)  and ($|\phi (Y)| > 2 + t_1 + \alpha$ or $|\phi (X)| > H(y) + t_1$). Thus, if $\phi'$ is sufficiently close to $\phi$, the same inequalities hold when we replace $\phi$ with $\phi'$. Hence $e \notin T'(t_1)$. This proves the claim, since there are only finitely many edges in $T(t_2) \setminus T(t_1)$.

We know that $T(m(\Um)-2-\alpha)$ is a non-empty subtree of $T(t_2)$ and $t_1 > m(\Um) - 2 - \alpha$, it follows that, for $\phi'$ sufficiently close to $\phi$, we have $T'(t_1) \cap T(t_2) \supseteq T(m(\Um)-2-\alpha) \neq \emptyset$.

Since $T'(t_1) \cap T(t_2) \subseteq T(t_1)$ and $T(t_1)$ is contained in the interior of $T(t_2)$, we know that $T(t_2)$ contains a connected component of $T'(t_1)$. Since $T'(t_1)$ is connected, we must have $T'(t_1) \subseteq T(t_2)$. Therefore $T'(t_1)$ is finite, and so $\phi' \in ({\bf \Phi}_{\Um})_{Q}$.
\end{proof}

\begin{Theorem}\label{thm:properlydiscontinuously}
The mapping class group acts properly discontinuously on the set $({\bf \Phi}_{\Um})_{Q}$.
\end{Theorem}

\begin{proof}
The statement is equivalent to the fact that $\mathrm{PSL}(2,\Z)$
acts properly discontinuously, that is, for any compact subset $K$
of $({\bf \Phi}_{\Um})_{Q}$, the set $\{H \in \mathrm{PSL}(2,\Z)
\mid H K\cap K \neq \emptyset \}$ is finite. Suppose not, then there
exists a sequence of distinct $H_i \in \mathrm{PSL}(2,\Z)$ and
$\phi_i \in K$ such that $H_i(\phi_i) \in K$. Passing to a
subsequence, by the compactness of $K$, we may assume that $\phi_i
\rightarrow \phi \in K$, $H_i \rightarrow \infty$, and $H_i(\phi_i)
\rightarrow \phi' \in K$. (Note that in a discrete group, for
example $\mathrm{PSL}(2,\Z)$, an infinite sequence of distinct
elements $H_i$ should tend to $\infty$.)

Now, as in the proof of Theorem \ref{thm:openproperty}, we have the tree $T_{\phi}(t_1)$ of $\phi$ is finite for some $t_1>0$, and that $T_{\phi_i}(t_1)$ is contained in the finite tree $T_{\phi}(t_2)$, for some $t_2>t_1$ and for all $i$ sufficiently large. This implies that the same constant $\kappa$ can be used in the lower Fibonacci bound for all $\phi_i$ for $i$ sufficiently large, and hence $H_i(\phi_i) \rightarrow \infty$ as $i \rightarrow \infty$. (Note that, in order to make sense of $\phi \rightarrow \infty$, we use the identification of ${\bf \Phi}_{\Um}$ with $\mathcal{X}_{\Ut}$, where $\Um = \mathrm{GT} (\Ut)$, and hence with the character variety $\mathcal{V} = \{(x,y,z) \in {\C}^3 \mid x^2+y^2+z^2+xyz=px +qy +rz +s\}$; see the discussion in the Introduction and in Remark \ref{phi_ut}.) This contradicts $H_i(\phi_i) \rightarrow \phi' \in K$.
\end{proof}

\section{The real case}\label{s:real_case}

In this section we focus on the set $\mathcal{X}_{\Ut}^\R$ of real characters $(x,y,z)$ with boundary data $\Ut = (a,b,c,d) \in \R^4$. The representations corresponding to these characters are representations of $\pi_1 (S)$ into one of the two real forms of $\SLtwoC$, namely $\SLtwoR$ and $\mathrm{SU}(2)$. In particular, we will prove Theorem \ref{thmB} from the Introduction, see Theorem \ref{thm:realnonemptydomain} and Corollary \ref{real_domain}.

\subsection{Topology of the real character variety}

In \cite{ben_the}, Benedetto and Goldman completely described the topology of the set $\mathcal{X}_{\Ut}^{\mathbb{R}}$, when $\Ut \in \R^4$. There are six different cases depending on the number $n$ of boundary traces in $[-2 , 2]$, that can vary from $0$ to $4$:

\begin{enumerate}
    \item a quadruply punctured sphere, if $n =0$ and $abcd<0$;
    \item a disjoint union of a triply punctured torus and a disc, if $n=0$ and $abcd>0$;
    \item a disjoint union of a triply punctured sphere and a disc, if $n=1$;
    \item a disjoint union of an annulus and two discs, if $n=2$;
    \item a disjoint union of four discs, if $n = 3$;
    \item a disjoint union of four discs and a sphere, if $n =4$.
\end{enumerate}

Representations in $\mathrm{SU}(2)$ are such that $n = 4$ and correspond to the compact connected component of $\mathcal{X}_{\Ut}^\R$. However, when the parameters $\Ut = (a,b,c,d) \in [-2 , 2]$ satisfy certain inequalities (see Proposition 1.4 of \cite{ben_the}), the compact component consists of representations in $\SLtwoR$. The action of the mapping class group on the $\mathrm{SU}(2)$-character variety is ergodic by the work of Goldman \cite{gol_erg}.

The non-compact components always correspond to representations into
$\SLtwoR$, and in this case the dynamics of the action is richer.
When $n=0$ the representations send boundary curves to hyperbolic
elements of $\SLtwoR$. In this case, by Goldman \cite{gol_top} the
components are indexed by the relative Euler class which, according
to Milnor-Wood inequality, is an element of $\{-2,-1,0,1,2\}$.

So, in the case $n=0$, we have the following identifications:
\begin{itemize}
\item The disc component when $abcd>0$ corresponds to hyperbolic structures on $S$ with geodesic boundary or equivalently to representations with maximal relative Euler class $2$ or $-2$; see \cite{ben_the}.
\item The component in the case $abcd<0$ corresponds to representations with relative Euler class $1$ or $-1$. Indeed, using the additivity of the relative Euler class, when we cut $S$ into two pairs of pants, the relative Euler class of one of them is 0 and the other is $\pm 1$. Since the product of the three boundary traces  is positive in the first pair of pants and negative in the other, we have $abcd<0$.
\item The triply punctured torus component corresponds to representations with relative Euler class $0$. For the same reason, when cutting $S$ into two pairs of pants the two relative Euler class are $+1$ and $-1$, and hence both products of boundary traces are negative.
\end{itemize}

In the next section we will prove that the action of the mapping class group on the components corresponding to non-maximal Euler class are never ergodic on the whole component, except for a compact subset of dimension 1 of parameters $\Ut = (a,b,c,d)\in \R^4$ corresponding to the equations $p(a,b,c,d) = q(a,b,c,d)=r(a,b,c,d) = 0$ and $s\in [4 , 20]$. This is rather surprising as one could have expected that when $(p,q,r)$ are small enough, the dynamics of the action would be very close to the dynamics on the one-holed torus, which is known to be ergodic. In fact, domains of discontinuity will appear as soon as these parameters are non-zero; see Theorem \ref{thm:realnonemptydomain} and Corollary \ref{real_domain}.

\subsection{Construction of a Markoff-map with BQ-conditions}

We will prove in this section that, if the parameters $p$, $q$ and
$r$ are not all zero, then the set of characters in
$\mathcal{X}_{\Ut}^\R$ satisfying the BQ-condition is non-empty.

\begin{Lemma} Let $\Um = (p,q,r,s)$ and suppose $(p,q,r) \neq (0,0,0)$.

For all $\varepsilon >0$ and $K >0$, there exists a real
$\Um$--Markoff triple $(x_1,x_2,x_3)$ such that one of the values is
$-( 2+ \varepsilon)<x_i<  -2$ and the other two satisfy  $x_j  x_k >
0$ and  $|x_j| , |x_k| > K$
\end{Lemma}

\begin{proof}
As $(p,q,r)\neq (0,0,0)$, we may assume, without loss of generality, that $(q,r) \neq (0,0)$ with $qr \geq 0$  (so $q$ and $r$ have the same sign). In this case, we search a $\Um$--Markoff triple of the form $(- (2+\epsilon) ,   y ,  y)$ with $\epsilon > 0$ and $y$ is such that $y (q+r) < 0$ (so $y$ and $q$ are of opposite sign).

Such a point satisfies the equation:
$$(2+\epsilon)^2 + 2y^2 - (2+\epsilon)y^2 + p (2+\epsilon) - (q+r) y - s = 0.$$
This equation is quadratic in $\epsilon$ so we can express
$\epsilon$ as a function of $y$:

$$\epsilon^\pm(y) = \frac 12 \left( (y^2 - p - 4) \pm \sqrt{ y^4 -8y^2-2p y^2 +4 (q+r)y + (p^2+4s)} \right).$$

When $y$ gets large and we choose the solution with the minus sign, we have the following Taylor series.

\begin{align*}
\epsilon(y) & \underset{y \rightarrow \infty }{=} \frac 12 \left(  (y^2 -p -4) - y^2 \sqrt{ 1+\left(\frac{-8-2p}{y^2} +4 \frac{q+r}{y^3} + \frac{p^2+4s}{y^4} \right) } \right) \\
 &  \underset{y \rightarrow \infty }{=} \frac 12 \left( (y^2 -p -4) - y^2 \left( 1+\frac 12 \left(\frac{-8-2p}{y^2} +4 \frac{q+r}{y^3} \right) + o\left(\frac{1}{y^3}\right) \right) \right)\\
&\underset{y \rightarrow \infty }{=} - \dfrac{q+r}{y} + o\left( \frac{1}{y} \right).
\end{align*}

So taking $y$ large enough (in particular larger than $K$), we have $0<\epsilon(y)<\varepsilon$ and $|y| >K$. Hence the triple $(-2 -\epsilon(y) , y , y)$ is a $\Um$--Markoff triple satisfying the conditions.
\end{proof}

\begin{Lemma}
Let $\Um = (p,q,r,s)$ and suppose $(p,q,r) \neq (0,0,0)$. There
exists a $\Um$--Markoff map $\phi \in {\bf \Phi}_{\Um} $ such that
$|\Omega_\phi (2+\alpha)| $ is finite.
\end{Lemma}

\begin{proof}
As the three numbers $p$, $q$ and $r$ are not all zero, there are two of them that are of the same sign and not both zero. Without loss of generality, we can assume that $(q,r)\neq (0,0)$ and $q \leq 0$ and $r \leq 0$ (the positive case can be treated the same way). We know from the previous lemma that we can find a $\Um$--Markoff triple $(-2-\epsilon(y) , y , y)$, with $y >0$ large, and  $\epsilon(y)$ is a function of $y$ whose Taylor series is $\epsilon(y) = -\frac{q+r}{y} + o\left( \frac{1}{y} \right)$.  Let $\phi$ be the corresponding $\Um$--Markoff map. Let's call $X_0 , Y_0$ and $Z_0$ the regions corresponding to this triple. The neighbors around $X_0$ are denoted by the two sequences $Y_n$ and $Z_n$.

The values of the neighbors of $X_0$ are of the form:
\begin{align*}
    y_{n+1} & = q + (2+\epsilon) z_n - y_n; \\
    z_{n+1} & = r + (2+\epsilon) y_{n+1} - z_n.
\end{align*}

$\bullet$ First assume that $q$ and $r$ are both non-zero so that $-(q+r) > \max \{-q , -r\}$. Hence, we can choose $y_0 = z_0 = y$ large enough so that
$$\epsilon = \epsilon(y) > \frac{\max \{-q , -r\}}{y} > 0.$$
In this case, we can prove by recurrence that $ \forall n > 0$ we have:
$$y_{n+1} > z_n > y_n > y.$$
Indeed for $n=1$ we have:
\begin{align*}
    y_1 & = q + (2+\epsilon)y - y = y + \epsilon y + q > y;   \\
    z_1 & = r + (2+\epsilon)y_1 - y = y_1 + (y_1-y) + r + y_1 \epsilon >  y_1; \\
    y_2 & = q + (2+ \epsilon)z_1 - y_1 > z_1 + (z_1-y_1) + \epsilon z_1 + q  > z_1.
\end{align*}
The induction follows the same steps:
\begin{align*}
    z_{n+1} & = r + (2+\epsilon)y_{n+1} - z_n = y_{n+1} + (y_{n+1} - z_n)  + y_{n+1}\epsilon + r \\
        & >  y_{n+1} + y \epsilon+r \\
        & > y_{n+1}; \\
    y_{n+2} & = q + (2+ \epsilon)z_{n+1} - y_n > z_{n+1} + (z_{n+1}-y_{n+1}) + z_{n+1} \epsilon + q \\
        & > z_{n+1}.
\end{align*}
A similar treatment can be applied to $n < 0$, in which case we have $z_{n-1} > y_{n} > z_{n} > y$.

This proves that all neighbors of $X_0$ have value greater than $y$. As we can choose $y$ to be greater than $2+\alpha$, and the set $\Omega_\phi(2+\alpha)$ is connected, it is clear that $X_0$ is the only region in $\Omega_\phi(2+\alpha)$.

$\bullet$ Now assume one of $q$ or $r$ is zero, for example $r = 0$.  We can choose $y$ large enough so that $  \epsilon = \epsilon(y)  > -\frac 34 \frac{q}{y} $ and $\epsilon < \frac 14 $. In this case we can prove that $ \forall n \geq  2$ we have:
$$y_{n+1} > z_n > y_n - \frac 34 q \mbox{     and     } y_n \geq y.$$

Indeed for the first terms of the sequence, we have:
\begin{align*}
    y_1 & =  q + (2+\epsilon)y - y = y + (q+ \epsilon y) > y+\frac q4;    \\
    z_1 & =  (2+\epsilon)y_1 - y \\
        &  >  y_1 + (y_1 - y) + y_1 \epsilon > y_1 +  \frac q4 -   \frac{3q}{4}  + \epsilon \frac q4 \\
        & > y_1 - \frac q2; \\
    y_2 & =  q + (2+ \epsilon)z_1 - y_1 > z_1 + (z_1-y_1) + \epsilon (y) z_1 > z_1; \\
    z_2 & = (2+ \epsilon) y_2 - z_1 = y_2 + (y_2 - z_1) +  \epsilon y_2 > y_2 - \frac{3q}{4}; \\
    y_3 & = q + (2+\epsilon) z_2 - y_2 > z_2 -\frac q2.
\end{align*}
The induction follows the same steps.
\begin{align*}
    z_{n+1} & =  (2+\epsilon)y_{n+1} - z_n = y_{n+1} + (y_{n+1} - z_n)  + y_{n+1}\epsilon \\
        & >  y_{n+1} + y \epsilon> y_{n+1} - \frac{3q}{4}; \\
    y_{n+2} & = q + (2+ \epsilon)z_{n+1} - y_n > z_{n+1} + (z_{n+1}-y_{n+1}) + z_{n+1} \epsilon + q \\
        & > z_{n+1} - \frac{3q}{4} - \frac{3q}{4} + q \\
        & > z_{n+1}.
\end{align*}
Similarly for $n\leq -2$, we have the same results. In conclusion in all cases, the set $\Omega_\phi (2+\alpha)$ is finite.
\end{proof}

To conclude, it remains to show that the Markoff map $\phi$ constructed in the previous lemma satisfies the BQ-conditions. First, note that $\Omega_\phi (2)$ is empty, as we constructed an attracting subtree such that the value of all regions adjacent to this tree are greater than 2.  Now, using the proof of Theorem \ref{thm:B2}, we see that, if  $\Omega_\phi (2+ \alpha)$ is finite and $\phi^{-1} ( [-2 , 2]  \cup \mathcal{S}_{\Um} ) = \emptyset$, then $\log^+ |\phi|$ has Fibonacci growth, and hence $\Omega_\phi (L)$ is finite. In addition, since in the construction of $\phi$ we can choose any $y>0$ large enough as a starting point, and since $\mathcal{S}_{\Um}$ is finite,  we can see that we can choose $\phi$ such that $\phi^{-1} ( \mathcal{S}_{\Um} ) = \emptyset$. 

We can now state the theorem.

\begin{Theorem}\label{thm:realnonemptydomain}
Let $\Um = (p,q,r,s)$ and suppose $(p,q,r) \neq (0,0,0)$. Then the
set $({\bf \Phi}_{\Um})_{Q}^{\R} $ is non-empty.
\end{Theorem}

\subsection{Domain of discontinuity}

The result of the previous section implies the following result on the action of
the mapping class group on the relative character variety.

\begin{Corollary} \label{real_domain}
Let $\Ut = (a,b,c,d) \in \R^4$. The action of the mapping class
group is ergodic on the whole real slice of the relative character
variety $\mathcal{X}_{\Ut}^\R $ if and only if $|a| = |b| = |c| =
|d|$ with $abcd \leq 0$ and $ 2 < |a| < \sqrt{2(1+\sqrt{5})}$ or
$a=0$.

Otherwise, there exists an open domain of discontinuity.
\end{Corollary}

\begin{proof}
In the case $(p,q,r)= (0,0,0)$, the set of $(0,0,0,s)$--Markoff maps
is identified with the set of $s$--Markoff maps in the sense of
Tan-Wong-Zhang. These maps correspond to real character of the
one-holed torus. The action of the mapping class group on the real
character is completely described by Goldman, and can be stated as
follows:
\begin{itemize}
    \item If $s<0$, there are four contractible connected components in the character variety and the action is properly discontinuous on each one.
    \item If  $s\in [0 , 4[$, there are four contractible connected components on which the action is properly discontinuous, and one compact component where the action is ergodic.
    \item If $s \in [4 , 20]$ there is a unique connected component and the action of the mapping class group is ergodic.
    \item If $s>20$, there is a unique connected component. There is an open domain $\mathcal{D}$ which is a domain of discontinuity for the action, and on the complementary of $\mathcal{D}$ the action is ergodic.
\end{itemize}
So the only case where the action is ergodic on the whole character variety is when $s \in [4 , 20]$.

From Goldman-Toledo \cite{gol_aff}, we have  $(p,q,r)=(0,0,0)$ if and only if one of the following is satisfied:
\begin{enumerate}
    \item $|a| = |b| = |c| = |d|$ with $abcd \leq 0$
    \item Three of $a,b,c,d$ are zero.
\end{enumerate}
In the first case $s = 4-a^2-b^2-c^2-d^2-abcd = 4 - 4a^2+a^4$. Solving the system $4\leq 4 - 4a^2 + a^4 \leq 20$ gives
$$2 \leq |a| \leq \sqrt{2(1+\sqrt{5})} \mbox{  or   } a=0$$
In the second case, we can assume that $b=c=d=0$. Then we have that $s = 4 - a^2\leq 4$. Hence $s\in [4 , 20]$ implies that $a=0$ and hence we have $a=b=c=d=0$ which is included in the first case.

Otherwise, when $(p,q,r)$ is non zero, there is a non-empty open
domain of discontinuity by Theorem \ref{thm:realnonemptydomain} and
hence the action is not ergodic.
\end{proof}

\section{Concluding remarks}\label{s:Conclusion}

This paper was mostly concerned with describing a domain of discontinuity for the action of $\MCG(S)$ on the (relative) character variety of the four-holed sphere $S$. We expect that it would be possible to implement this on a program to draw various slices of the domain of discontinuity, and we expect that, in general, these slices would have highly fractal boundary, as in the case of the character variety for the free group $F_2$ on two generators analysed by Series, Tan and Yamashita \cite{ser_xxx}. Restriction to the real case may also be of interest. It is not clear at this point if the domains of discontinuity for the real case would have relatively smooth boundary as in the case analysed by Goldman in \cite{gol_the} for $F_2$, or if they would exhibit fractal type boundary.

There are also other group actions on the character variety which are of interest. In particular, we can consider the action of $Out(F_3)$ as studied by Minsky \cite{min_ond} (this is a bigger group action than what we consider). In this case, it is somewhat difficult to give easily (computer) verifiable conditions to describe the domain of discontinuity like the BQ-conditions, nonetheless, our methods may provide a starting point towards this end.

Finally we note that, as pointed out in the beginning of the
Introduction, much of the original motivation for the paper came
from McShane's identity and the proof provided by Bowditch. A new
generalization to the identity from the point of view of Coxeter
group actions on quartic varieties was given by Hu, Tan and Zhang in
\cite{hut_cox}, and we plan to pursue this direction in a future
paper and explore possible generalizations of the McShane's identity
in our context.

\bibliographystyle{plain}

\bibliography{sample2}

\end{document}